\documentclass[notitlepage]{article}

 \usepackage[matrix,arrow,curve]{xy}
 \usepackage{graphicx}
 \usepackage[english]{babel}
 \usepackage{amsmath}
 \usepackage{amssymb, indentfirst}
 \usepackage{amsthm}
 \usepackage{enumitem}
 \usepackage{tikz}
 \usepackage{MnSymbol}
\newcommand{\Hom}{{\rm Hom}}
\newcommand{\kk}{\textbf{k}}
\renewcommand{\mod}{{\rm mod}\text{-}}

\newcommand{\sstop}{{\rm s\text{-}top}}
\newcommand{\ssrad}{{\rm s\text{-}rad}}

\newcommand{\rmod}{{\rm mod} \text{-}}
\newcommand{\smod}{{\rm\underline{mod}} \text{-}}
\newcommand{\DTr}{{\rm DTr}}
\newcommand{\soc}{{\rm soc}}
\newcommand{\rad}{{\rm rad}}
\newcommand{\Imm}{{\rm Im}}
\newcommand{\Ker}{{\rm Ker}}
\newcommand{\Coker}{{\rm Coker}}
\newcommand{\Cone}{{\rm Cone}}
\newcommand{\topp}{{\rm top}}

\newcommand{\Ends}{{\rm\underline{End}}}
\newcommand{\Homs}{{\rm\underline{Hom}}}
\newcommand{\Ext}{{\rm Ext}}
\newcommand{\rn}{\romannumeral }
\newcommand{\scl}{{\rm sc}}
\newtheorem{cor}{Corollary}
\newtheorem{lem}{Lemma}
\newtheorem{rem}{Remark}

\newtheorem{opr}{Definition}
\newtheorem{thm}{Theorem}
\newtheorem{prop}{Proposition}

\begin{document}

\title{On stably biserial algebras and the Auslander-Reiten conjecture for special biserial algebras}
\author{Mikhail Antipov\thanks{Mikhail Antipov was partially supported by grant NSh-9721.2016.1 of the President of the Russian Federation.} \hspace{1pt} and Alexandra Zvonareva\thanks{Alexandra Zvonareva was supported by the RFFI Grants 16-31-60089 and 16-31-00004.}}
\date{}
\maketitle

\begin{abstract}
By a result claimed by Pogorza\l{}y selfinjective special biserial
algebras can be stably equivalent only to stably biserial algebras
and these two classes coincide. By an example of Ariki, Iijima and
Park the classes of stably biserial and selfinjective special
biserial algebras do not coincide. In these notes we provide a
detailed proof of the fact that a selfinjective special biserial
algebra can be stably equivalent only to a stably biserial algebra
following some ideas from the paper by Pogorza\l{}y. We will analyse
the structure of symmetric stably biserial algebras and show that in
characteristic $\neq 2$ the classes of symmetric special biserial
(Brauer graph) algebras and symmetric stably biserial algebras
indeed coincide. Also, we provide a proof of the Auslander-Reiten conjecture for special biserial algebras. \end{abstract}

\section{Introduction}
Derived equivalences of symmetric special biserial or equivalently
Brauer graph algebras \cite{Sch} have been extensively studied over
the past few years \cite{AntD, Kau, RZ, RS, Zvo2, VZ, IMA, Zim, MS,
AZ, Zvo1, SZI, Aih, AAC}. These studies concern mainly attempts to
classify symmetric special biserial algebras up to derived
equivalence, classification of special tilting complexes over such
algebras or computation of the derived Picard groups. It is
well know that the class of symmetric special biserial algebras of
finite representation type is closed under derived equivalence. The
fact that the class of symmetric special biserial algebras is closed
under derived equivalence followed from the results of Pogorza\l{}y
\cite{Pog}. Unfortunately, in \cite{AIP} counterexamples for some of
the statements of \cite{Pog} were given.

In this paper we reprove the fact that if a selfinjective
algebra (not isomorphic to the Nakayama algebra with $\rad^2=0$) is stably equivalent to a selfinjective special biserial
algebra, then it is stably biserial. We do not use the original
approach of Pogorza\l{}y via Galois coverings, instead we perform all
combinatorial computations directly. We give a proof of the Auslander-Reiten conjecture for special biserial algebras using the reduction to the selfinjective case obtained by Mart\'inez-Villa. The conjecture states that the number of non-isomorphic non-projective simple modules is invariant under stable equivalence. The proof for selfinjective special biserial algebras in more involved, since we have to consider systems of orthogonal stable bricks over stably biserial algebras. After that we describe all
symmetric stably biserial algebras, showing that in characteristic
$\neq 2$ this class indeed coincides with the class of symmetric
special biserial algebras. This is the first step towards the proof
of the fact that the class of symmetric special biserial algebras is
closed under derived equivalence.

\section{Preliminaries}

Throughout this paper $A$ is a basic, connected, finite dimensional
algebra over an algebraically closed field $\kk$, $\rmod A$ is the
category of finite-dimensional right $A$-modules, $\smod A$ is the stable category of
$\rmod A$, i.e. the category of modules modulo the maps factoring through projective modules. In the case where $A$ is selfinjective the category $\smod A$ is triangulated. The Auslander-Reiten translation $\DTr$ will be denoted by $\tau$, the $\Hom$-spaces in $\smod A$ will be denoted by
$\Homs$, for $f \in \rmod
A$ its class in $\smod A$ will be denoted by $\underline{f}$, the syzygy or the Heller's
loop functor will be denoted by $\Omega:
\smod A \rightarrow \smod A$. A module will be called local, if it is an epimorphic
image of an indecomposable projective module.

\begin{opr}\label{stb} Let $Q$ be a quiver, $I$ an admissible ideal of $\kk Q$. A
selfinjective algebra $A'$ is called stably biserial if it is isomorphic to $A = \kk Q/I$, where $Q$ and $I$ satisfy the following conditions:

(a) For each vertex $i \in Q$, the number of outgoing arrows and the
number of incoming arrows are less than or equal to 2;

(b) For each arrow $\alpha \in Q$, there is at most one arrow $\beta
\in Q$ that satisfies $\alpha \beta \not\in \alpha \rad (A) \beta +
\soc (A)$;

(c) For each arrow $\alpha \in Q$, there is at most one arrow $\beta
\in Q$ that satisfies $\beta\alpha  \not\in  \beta\rad (A) \alpha +
\soc (A)$.
\end{opr}

The following description of stably biserial algebras was provided
in \cite{AIP}:

\begin{prop}[Proposition 7.5 \cite{AIP}] If $A$ is stably biserial then there exists a presentation of $A  \simeq \kk Q/I$ such that the following
conditions hold:

(1) If $\alpha\beta \neq 0$, $\alpha\gamma \neq 0$, $\beta \neq
\gamma$, for arrows $\alpha, \beta, \gamma$ then either $\alpha\beta
\in \soc (A)$ or $\alpha\gamma \in \soc (A)$;

(2) If $\beta\alpha \neq 0$, $\gamma\alpha \neq 0$, $\beta \neq
\gamma$, for arrows $\alpha, \beta, \gamma$ then either $\beta\alpha
\in \soc (A)$ or $\gamma\alpha \in \soc (A)$.
\end{prop}

\begin{opr}\label{ssb} Let $Q$ be a quiver, $I$ an admissible ideal of $\kk Q$. An
 algebra $A'$ is called special biserial if it is isomorphic to $A = \kk Q/I$, where $Q$ and $I$ satisfy the following conditions:

(a) For each vertex $i \in Q$, the number of outgoing arrows and the
number of incoming arrows are less than or equal to 2;

(b) For each arrow $\alpha \in Q$, there is at most one arrow $\beta
\in Q$ that satisfies $\alpha \beta \neq 0$;

(c) For each arrow $\alpha \in Q$, there is at most one arrow $\beta
\in Q$ that satisfies $\beta\alpha  \neq 0$.

If additionally $A$ is selfinjective, then it is called selfinjective special biserial.
\end{opr}

\section{Stable equivalences}

In this section we are going to prove that if an
algebra is stably equivalent to a selfinjective special biserial
algebra (not isomorphic to the Nakayama algebra with $\rad^2=0$), then it is stably biserial following the ideas from
\cite{Pog}.

\begin{prop}[Proposition 7.11 \cite{AIP}, see also Lemmas 5.3, 5.4 \cite{Pog}]\label{tau-period}
Let $B$ be an indecomposable selfinjective algebra which is not a
local Nakayama algebra. Then, we have the following:

(1) If $P$ is indecomposable projective, then $\tau(P/\soc (P))
\not\simeq P/\soc(P)$;

(2) If $S$ is simple, then $S$ is non-projective and $\tau(S)
\not\simeq S$.
\end{prop}

From now on we are not going to consider local Nakayama algebras.
Thus, we can assume that $A$ does not have any simple
modules of $\tau$-period 1.

\begin{opr}
Let $A$ be a selfinjective $\kk$-algebra. An indecomposable $A$-module
$M$ is said to be a stable brick if $\Ends(M) \simeq \kk$. A family
$\{M_i\}_{i \in I}$ of mutually non-isomorphic stable bricks is a system of
orthogonal stable bricks if the following conditions hold:

(1) $M_i$ is not of $\tau$-period $1$ for every $i \in I$;

(2) $\Homs(M_i, M_j) = 0$ for any $i, j \in I$ with $i \neq j$.

A system of orthogonal stable bricks $\{M_i\}_{i \in I}$ is called
maximal if for every indecomposable $A$-module N that is neither
projective nor of $\tau$-period $1$ there exist $i, j \in I$ such
that $\Homs(M_i,N) \neq 0$ and $\Homs(N,M_j) \neq 0$.
\end{opr}

\begin{rem}
If there is an equivalence of categories $\smod B
\rightarrow \smod A$, where $A$ and $B$ are selfinjective, then the
image of the set of representatives of the iso-classes of simple modules is a maximal system of orthogonal stable
bricks.

Since we are interested in maximal systems of orthogonal stable bricks
which are images of the sets of simple modules, for now we can
assume, that the cardinality of $\mathcal{M}$ is finite.
\end{rem}

\begin{opr}
Let $\mathcal{M} = \{M_1,\cdots, M_n\}$ be a maximal system of
orthogonal stable bricks. An indecomposable $A$-module $N$ is called
s-projective with respect to $\mathcal{M}$ if the following
conditions are satisfied:

(1) $N$ is not of $\tau$-period $1$;

(2) $\Homs(N,\oplus_{i=1}^n M_i) \simeq \kk$;

(3) If $\Homs(N,M_i)\neq 0$, then for every non-zero $\underline{f}
: X \rightarrow M_i$ and  $\underline{g} : N
\rightarrow M_i$ there exists $\underline{h} : N \rightarrow X$ such
that $\underline{fh} = \underline{g}$.

An $A$-module $N$ is called s-projective with respect to
$\mathcal{M}$ if it is a sum of indecomposable s-projective modules;
s-injective modules are defined dually.
\end{opr}

It is clear that for an indecomposable s-projective $A$-module $N$
there exists only one $i\in I$ such that $\Homs(N,M_i)\neq 0$. In
\cite{Pog2} it is proved that an indecomposable $A$-module $N$ is
s-projective with respect to $\mathcal{M}$ if and only if $N \simeq
\tau^{-1}\Omega(M)$ for some $M \in \mathcal{M}$. Let $N$ be an
 indecomposable s-projective $A$-module with respect to $\mathcal{M}$. We say that
$\sstop(N) \simeq M$ if $M \in \mathcal{M}$  and
$\Homs(N,M) \neq 0$. In this case $\sstop(\tau^{-1}\Omega(M)) \simeq
M$ for $M \in \mathcal{M}$. See also  \cite[Proposition 7.13]{AIP}.

\begin{rem}
If there is an equivalence of categories $\smod B \rightarrow \smod
A$, where $A$ and $B$ are selfinjective and $\mathcal{M} =
\{M_1,\cdots, M_n\}$ is the image of the set of simple $B$-modules, then
the image of the module of the form $P/\soc(P)$, where $P$ is an
indecomposable projective $B$-module, is indecomposable s-projective with respect to
$\mathcal{M}$.
\end{rem}

We will denote by $Q_0$ the set of vertices of $Q$, by $Q_1$ the set of arrows of $Q$ and by  $s(\alpha)$, $e(\alpha)$ the maps from $Q_1$ to $Q_0$, which map an arrow to its beginning and end respectively. 

From now on, when considering a selfinjective special biserial algebra $A=\kk Q/I$ we will fix a presentation satisfying the conditions from Definition \ref{ssb}. Note that the generating set of relations in $I$ can be chosen to consist of relations of three kinds: zero relations $\alpha\beta=0$ for some $\alpha, \beta \in Q_1$; relations of the form $\alpha_1\cdots\alpha_m=c\beta_1\cdots\beta_n$ ($c\in \kk^*$) for $\alpha_1\neq \beta_1$ and $s(\alpha_1)=s(\beta_1)$; relations of the form $\alpha_1\cdots\alpha_m=0$ in the case when there is only one arrow leaving $s(\alpha_1)$ ($\alpha_i, \beta_j \in Q_1$).

Recall that an indecomposable non-projective module over a
selfinjective special biserial algebra $A=\kk Q/I$ is either a string
or a band module. Since all the band modules are of $\tau$-period $1$ we
are not going to use them. 

Given an arrow $\alpha \in Q_1$, we will denote by $\alpha^{-1}$ its
formal inverse; thus $s(\alpha^{-1})=e(\alpha)$,
$e(\alpha^{-1})=s(\alpha)$, $(\alpha^{-1})^{-1}=\alpha$. The set of
formal inverse arrows $\{\alpha^{-1}\}_{\alpha \in Q_1}$ will be
denoted by $Q_1^{-1}$. A string of length $n$ is a sequence of the form
$c=c_1\cdots c_n$, where $c_i \in Q_1 \cup Q_1^{-1}$,
$s(c_{i+1})=e(c_i)$, $c_i \neq c_{i+1}^{-1}$ and neither $c_i\cdots
c_{i+t}$ nor $c_{i+t}^{-1}\cdots c_{i}^{-1}$ belong to $\soc (A)$ for any $i$ and $t$.
Additionally, for every vertex $x \in Q_0$, there is a string of length
zero denoted by $1_x$ with $s(1_x)=e(1_x)=x$. For a string
$c=c_1\cdots c_n$ of positive length, let $s(c):=s(c_1)$, $e(c):=e(c_n)$.

Let $c=c_1\cdots c_n$ be a string of length $n \geq 1$. A string
module $M_c$ is defined as follows: fix a basis $\{z_0, \cdots,
z_n\}$, given an idempotent $e_x$, corresponding to the vertex $x$,
$z_ie_x=z_i$ if 
 $x=e(c_{i})$ or  $x=s(c_{i+1})$ and zero otherwise.
Given an arrow $\alpha \in Q_1$, $z_i\alpha=z_{i-1}$ if
$c_{i}=\alpha^{-1}$, $z_i\alpha=z_{i+1}$ if $c_{i+1}=\alpha$ and
zero otherwise. To the string of length zero $1_x$ we associate the
simple module corresponding to the vertex $x$. Two string modules
corresponding to different strings $c$ and $c'$ are isomorphic if
and only if $c=c_1\cdots c_n$ and $c'=c_n^{-1} \cdots c_1^{-1}$.
Usually we will depict the string and the corresponding module by
the diagram of that module, e.g., the string
$\alpha^{-1}\beta\gamma\delta^{-1}$ will be depicted as
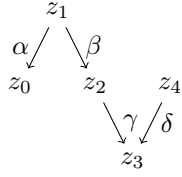
\begin{center}
\begin{tikzpicture}[xscale=0.5]
\node (v1) at (-1,2) {$z_1$}; \node (v2) at (-2,1) {$z_0$}; \node
(v3) at (0,1) {$z_2$}; \node (v4) at (1,0) {$z_3$}; \node (v5) at
(2,1) {$z_4$}; \draw [->] (v1) edge [left] node{$\alpha$} (v2) ;
\draw [->] (v1) edge [right] node{$\beta$} (v3); \draw [->] (v3)
edge [right] node{$\gamma$} (v4); \draw [->] (v5) edge [right]
node{$\delta$} (v4);
\end{tikzpicture}
 \end{center}
We will call $z_i$ a peak if there is no $\alpha\in Q_1$ such that $z_{i-1}\alpha=z_i$ or $z_{i+1}\alpha=z_i$. We will call $z_i$ a deep if for all $\alpha\in Q_1$ we have $z_{i}\alpha=0$. In the example above $z_1, z_4$ are peaks and $z_0,z_3$ are deeps. Note that this is not the standard use of the terms peak and deep. In cases when it does not lead to confusion, we will omit the names of the arrows in the diagrams and we will use diagrammatic notation for the elements of the algebra $A$.

We shall now describe the Auslander-Reiten sequences in $\rmod A$,
containing string modules. The Auslander-Reiten sequences,
containing an indecomposable projective module $P$ in the middle term are of the form
$$0 \rightarrow \rad(P) \rightarrow \rad(P)/\soc(P) \oplus P\rightarrow P/\soc(P) \rightarrow 0. $$
Assume now that $M_c$ is a non-projective indecomposable module
not isomorphic to $P/\soc(P)$ for any projective module $P$. The
module $M_c$ is of the form

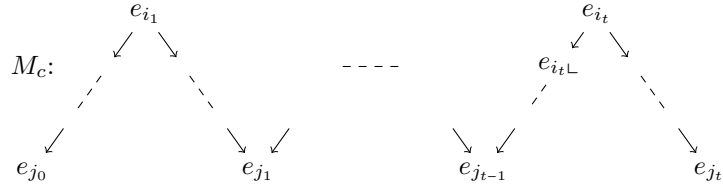
\begin{center}
\begin{tikzpicture}[xscale=0.5, yscale=0.7]
\node (v1) at (-2,1) {$e_{i_1}$}; \node (v2) at (-3,0) {}; \node
(v3) at (-4,-1) {}; \node (v4) at (-5,-2) {$e_{j_0}$}; \node (v5) at
(-1,0) {}; \node (v6) at (0,-1) {}; \node (v7) at (1,-2)
{$e_{j_1}$}; \node (v8) at (2,-1) {}; \node (v17) at (3,0) {}; \node
(v18) at (5,0) {}; \node (v9) at (6,-1) {}; \node (v10) at (7,-2)
{$e_{j_{t-1}}$}; \node (v11) at (8,-1) {}; \node (v13) at (9,0)
{$e_{i_t \lefthalfcup}$}; \node (v12) at (10,1) {$e_{i_t}$}; \node
(v14) at (11,0) {}; \node (v15) at (12,-1) {}; \node (v16) at
(13,-2) {$e_{j_t}$}; \draw [->] (v1) edge (v2); \draw [->] (v3) edge
(v4); \draw [->] (v1) edge (v5); \draw [->] (v6) edge (v7); \draw
[->] (v8) edge (v7); \draw [->] (v9) edge (v10); \draw [->] (v11)
edge (v10); \draw [->] (v12) edge (v13); \draw [->] (v12) edge
(v14); \draw [->] (v15) edge (v16); \draw [dashed] (v2) edge (v3);
\draw [dashed] (v5) edge (v6); \draw [dashed] (v17) edge (v18);
\draw [dashed] (v13) edge (v11); \draw [dashed] (v14) edge (v15);
\node at (-5,0) {$M_c$:};
\end{tikzpicture}
 \end{center}
where the first or the last directed substring may be trivial. Let
$c^r$ be the maximal string extending $c$ on the right by $e_{j_t}
\rightarrow e_{j_{t}\righthalfcup} \leftarrow \cdots \leftarrow
e_{i_{t+1}}$ if such a string exists (adding a co-hook).

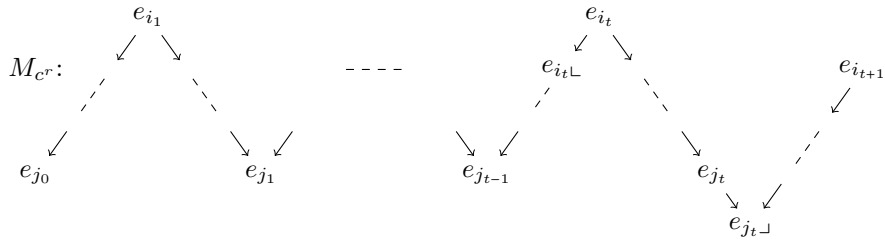
\begin{center}
\begin{tikzpicture}[xscale=0.5, yscale=0.7]
\node (v1) at (-2,1) {$e_{i_1}$}; \node (v2) at (-3,0) {}; \node
(v3) at (-4,-1) {}; \node (v4) at (-5,-2) {$e_{j_0}$}; \node (v5) at
(-1,0) {}; \node (v6) at (0,-1) {}; \node (v7) at (1,-2)
{$e_{j_1}$}; \node (v8) at (2,-1) {}; \node (v17) at (3,0) {}; \node
(v18) at (5,0) {}; \node (v9) at (6,-1) {}; \node (v10) at (7,-2)
{$e_{j_{t-1}}$}; \node (v11) at (8,-1) {}; \node (v13) at (9,0)
{$e_{i_t \lefthalfcup}$}; \node (v12) at (10,1) {$e_{i_t}$}; \node
(v14) at (11,0) {}; \node (v15) at (12,-1) {}; \node (v16) at
(13,-2) {$e_{j_t}$}; \draw [->] (v1) edge (v2); \draw [->] (v3) edge
(v4); \draw [->] (v1) edge (v5); \draw [->] (v6) edge (v7); \draw
[->] (v8) edge (v7); \draw [->] (v9) edge (v10); \draw [->] (v11)
edge (v10); \draw [->] (v12) edge (v13); \draw [->] (v12) edge
(v14); \draw [->] (v15) edge (v16); \draw [dashed] (v2) edge (v3);
\draw [dashed] (v5) edge (v6); \draw [dashed] (v17) edge (v18);
\draw [dashed] (v13) edge (v11); \draw [dashed] (v14) edge (v15);
\node at (-5,0) {$M_{c^r}$:}; \node (v19) at (14,-3) {$e_{j_t
\righthalfcup}$}; \node (v20) at (15,-2) {}; \node (v22) at (16,-1)
{}; \node (v21) at (17,0) {$e_{i_{t+1}}$}; \draw [->] (v16) edge
(v19); \draw [->] (v20) edge (v19); \draw [->] (v21) edge (v22);
\draw [dashed] (v22) edge (v20);
\end{tikzpicture}
 \end{center}

If not, let $c^r$ be the string obtained from $c$ by cancellation of
the last directed substring including the vertex $e_{i_t}$ ($c^r$ may be
empty),
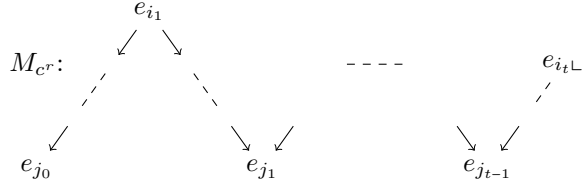
\begin{center}
\begin{tikzpicture}[xscale=0.5, yscale=0.7]
\node (v1) at (-2,1) {$e_{i_1}$}; \node (v2) at (-3,0) {}; \node
(v3) at (-4,-1) {}; \node (v4) at (-5,-2) {$e_{j_0}$}; \node (v5) at
(-1,0) {}; \node (v6) at (0,-1) {}; \node (v7) at (1,-2)
{$e_{j_1}$}; \node (v8) at (2,-1) {}; \node (v17) at (3,0) {}; \node
(v18) at (5,0) {}; \node (v9) at (6,-1) {}; \node (v10) at (7,-2)
{$e_{j_{t-1}}$}; \node (v11) at (8,-1) {}; \node (v13) at (9,0)
{$e_{i_t \lefthalfcup}$}; \draw [->] (v1) edge (v2); \draw [->] (v3)
edge (v4); \draw [->] (v1) edge (v5); \draw [->] (v6) edge (v7);
\draw [->] (v8) edge (v7); \draw [->] (v9) edge (v10); \draw [->]
(v11) edge (v10); \draw [dashed] (v2) edge (v3); \draw [dashed] (v5)
edge (v6); \draw [dashed] (v17) edge (v18); \draw [dashed] (v13)
edge (v11); \node at (-5,0) {$M_{c^r}$:};
\end{tikzpicture}
\end{center}
(deleting a hook). Similarly let $^{l}c$ be obtained from $c$ by the
corresponding operations on the left-hand side of $c$. Since $M_c$
is not isomorphic to $P/\soc(P)$ for any projective module $P$, at
least one of the strings $^l(c^r)$ or $(^lc)^r$ is non-empty, and if
both are defined, then $^l(c^r)=(^lc)^r$, let $^lc^r$ be the
non-trivial string $^l(c^r)$ or $(^lc)^r$. Then the Auslander-Reiten
sequence terminating at $M_c$ is of the form $$0 \rightarrow
\tau(M_c) \simeq M_{^lc^r} \rightarrow M_{c^r} \oplus M_{^lc}
\rightarrow M_c \rightarrow 0.$$ Similarly, $\tau^{-1}$ can be
computed by adding hooks if possible and deleting co-hooks if not
\cite{SW}, \cite{WW}, \cite{BR}.

The following lemma follows immediately from the
description of the Auslander-Reiten sequences.

\begin{lem}[see Lemma 6.4 \cite{Pog}]\label{s-proj}
Let $A$ be a selfinjective special biserial algebra and
let $\mathcal{M}$ be a maximal system of orthogonal stable bricks in
$\smod A$. Consider $M \in \mathcal{M}$ and let $N$ be an indecomposable
s-projective module with respect to $\mathcal{M}$ with $\sstop(N)
\simeq M$.

case (1): If $M$ is of the form
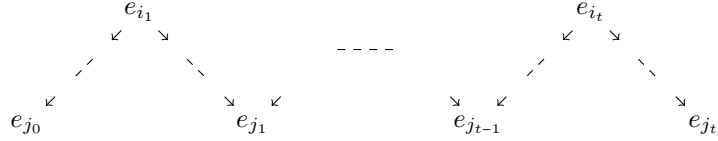
\begin{center}
\begin{tikzpicture}[xscale=0.5, yscale=0.5]
\node (v1) at (-2,1) {$e_{i_1}$}; \node (v2) at (-3,0) {}; \node
(v3) at (-4,-1) {}; \node (v4) at (-5,-2) {$e_{j_0}$}; \node (v5) at
(-1,0) {}; \node (v6) at (0,-1) {}; \node (v7) at (1,-2)
{$e_{j_1}$}; \node (v8) at (2,-1) {}; \node (v17) at (3,0) {}; \node
(v18) at (5,0) {}; \node (v9) at (6,-1) {}; \node (v10) at (7,-2)
{$e_{j_{t-1}}$}; \node (v11) at (8,-1) {}; \node (v13) at (9,0) {};
\node (v12) at (10,1) {$e_{i_t}$}; \node (v14) at (11,0) {}; \node
(v15) at (12,-1) {}; \node (v16) at (13,-2) {$e_{j_t}$}; \draw [->]
(v1) edge (v2); \draw [->] (v3) edge (v4); \draw [->] (v1) edge
(v5); \draw [->] (v6) edge (v7); \draw [->] (v8) edge (v7); \draw
[->] (v9) edge (v10); \draw [->] (v11) edge (v10); \draw [->] (v12)
edge (v13); \draw [->] (v12) edge (v14); \draw [->] (v15) edge
(v16); \draw [dashed] (v2) edge (v3); \draw [dashed] (v5) edge (v6);
\draw [dashed] (v17) edge (v18); \draw [dashed] (v13) edge (v11);
\draw [dashed] (v14) edge (v15);
\end{tikzpicture}
\end{center}
$t=0,1,\cdots$ then $N$ is of the form

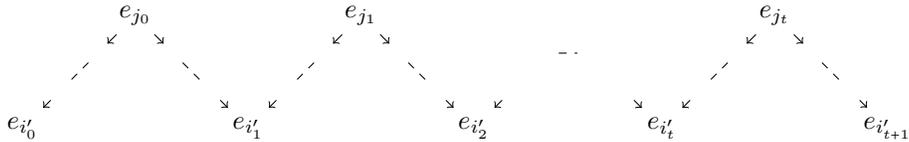
\begin{center}
\begin{tikzpicture}[xscale=0.5, yscale=0.5]
\node (v1) at (-3,2) {$e_{j_0}$}; \node (v2) at (-4,1) {}; \node
(v3) at (-5,0) {}; \node (v4) at (-6,-1) {$e_{i'_0}$}; \node (v5) at
(-2,1) {}; \node (v6) at (-1,0) {}; \node (v7) at (0,-1)
{$e_{i'_1}$}; \node (v8) at (1,0) {}; \node (v10) at (2,1) {}; \node
(v9) at (3,2) {$e_{j_1}$}; \node (v11) at (4,1) {}; \node (v12) at
(5,0) {}; \node (v13) at (6,-1) {$e_{i'_2}$}; \node (v14) at (7,0)
{}; \node (v23) at (8,1) {}; \node (v24) at (9,1) {}; \node (v15) at
(10,0) {}; \node (v16) at (11,-1) {$e_{i'_t}$}; \node (v17) at
(12,0) {}; \node (v19) at (13,1) {}; \node (v18) at (14,2)
{$e_{j_t}$}; \node (v20) at (15,1) {}; \node (v21) at (16,0) {};
\node (v22) at (17,-1) {$e_{i'_{t+1}}$}; \draw [->] (v1) edge (v2);
\draw [->] (v3) edge (v4); \draw [->] (v1) edge (v5); \draw [->]
(v6) edge (v7); \draw [->] (v8) edge (v7); \draw [->] (v9) edge
(v10); \draw [->] (v9) edge (v11); \draw [->] (v12) edge (v13);
\draw [->] (v14) edge (v13); \draw [->] (v15) edge (v16); \draw [->]
(v17) edge (v16); \draw [->] (v18) edge (v19); \draw [->] (v18) edge
(v20); \draw [->] (v21) edge (v22); \draw [dashed] (v20) edge (v21);
\draw [dashed] (v19) edge (v17); \draw [dashed] (v23) edge (v24);
\draw [dashed] (v2) edge (v3); \draw [dashed] (v5) edge (v6); \draw
[dashed] (v10) edge (v8); \draw [dashed] (v11) edge (v12);
\end{tikzpicture}
\end{center}
where $e_{j_{0}} \rightarrow \cdots \rightarrow e_{i'_{0}}$ and
$e_{j_{t}} \rightarrow \cdots \rightarrow e_{i'_{t+1}}$ are maximal directed
strings (may be trivial), $e_{i'_k} \leftarrow \cdots \leftarrow
e_{j_{k-1}}  \leftarrow \cdots \leftarrow e_{i_k}=c_ke_{i'_k}
\leftarrow \cdots \leftarrow e_{j_{k}}  \leftarrow \cdots \leftarrow
e_{i_k}$ in $A$ for $k=1,2,\cdots,t$ and some $c_k \in \kk^*$.

case (2): If $M$ is of the form
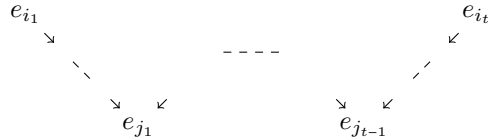
\begin{center}
\begin{tikzpicture}[xscale=0.5, yscale=0.5]
\node (v1) at (-2,1) {$e_{i_1}$}; \node (v5) at (-1,0) {}; \node
(v6) at (0,-1) {}; \node (v7) at (1,-2) {$e_{j_1}$}; \node (v8) at
(2,-1) {}; \node (v17) at (3,0) {}; \node (v18) at (5,0) {}; \node
(v9) at (6,-1) {}; \node (v10) at (7,-2) {$e_{j_{t-1}}$}; \node
(v11) at (8,-1) {}; \node (v13) at (9,0) {}; \node (v12) at (10,1)
{$e_{i_t}$}; \draw [->] (v1) edge (v5); \draw [->] (v6) edge (v7);
\draw [->] (v8) edge (v7); \draw [->] (v9) edge (v10); \draw [->]
(v11) edge (v10); \draw [->] (v12) edge (v13); \draw [dashed] (v5)
edge (v6); \draw [dashed] (v17) edge (v18); \draw [dashed] (v13)
edge (v11);
\end{tikzpicture}
\end{center}
$t=2,3,\cdots$ then $N$ is of the form

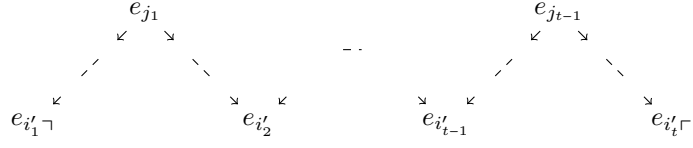
\begin{center}
\begin{tikzpicture}[xscale=0.5, yscale=0.5]
\node (v7) at (0,-1) {$e_{i'_1\righthalfcap} $}; \node (v8) at (1,0)
{}; \node (v10) at (2,1) {}; \node (v9) at (3,2) {$e_{j_1}$}; \node
(v11) at (4,1) {}; \node (v12) at (5,0) {}; \node (v13) at (6,-1)
{$e_{i'_2}$}; \node (v14) at (7,0) {}; \node (v23) at (8,1) {};
\node (v24) at (9,1) {}; \node (v15) at (10,0) {}; \node (v16) at
(11,-1) {$e_{i'_{t-1}}$}; \node (v17) at (12,0) {}; \node (v19) at
(13,1) {}; \node (v18) at (14,2) {$e_{j_{t-1}}$}; \node (v20) at
(15,1) {}; \node (v21) at (16,0) {}; \node (v22) at (17,-1)
{$e_{i'_{t}\lefthalfcap}$}; \draw [->] (v8) edge (v7); \draw [->]
(v9) edge (v10); \draw [->] (v9) edge (v11); \draw [->] (v12) edge
(v13); \draw [->] (v14) edge (v13); \draw [->] (v15) edge (v16);
\draw [->] (v17) edge (v16); \draw [->] (v18) edge (v19); \draw [->]
(v18) edge (v20); \draw [->] (v21) edge (v22); \draw [dashed] (v20)
edge (v21); \draw [dashed] (v19) edge (v17); \draw [dashed] (v23)
edge (v24); \draw [dashed] (v10) edge (v8); \draw [dashed] (v11)
edge (v12);
\end{tikzpicture}
\end{center}
where $e_{i_1} \rightarrow \cdots \rightarrow e_{j_1}\rightarrow \cdots \rightarrow
e_{i'_1\righthalfcap}$ and $e_{i_{t}} \rightarrow \cdots \rightarrow e_{j_{t-1}} \rightarrow
\cdots \rightarrow e_{i'_{t}\lefthalfcap}$ ($e_{j_1}$ may coincide
with $e_{i'_1\righthalfcap}$, $e_{j_{t-1}}$ may coincide with
$e_{i'_{t}\lefthalfcap}$) are maximal directed strings, $e_{i'_k} \leftarrow
\cdots \leftarrow e_{j_{k-1}}  \leftarrow \cdots \leftarrow
e_{i_k}=c_ke_{i'_k} \leftarrow \cdots \leftarrow e_{j_{k}}  \leftarrow
\cdots \leftarrow e_{i_k}$ in $A$ for $k=2,3,\cdots,t-1$ and $c_k \in \kk^*$.

case (3): If $M$ is of the form
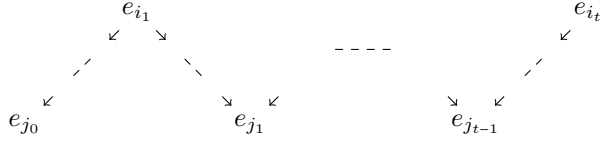
\begin{center}
\begin{tikzpicture}[xscale=0.5, yscale=0.5]
\node (v1) at (-2,1) {$e_{i_1}$}; \node (v2) at (-3,0) {}; \node
(v3) at (-4,-1) {}; \node (v4) at (-5,-2) {$e_{j_0}$}; \node (v5) at
(-1,0) {}; \node (v6) at (0,-1) {}; \node (v7) at (1,-2)
{$e_{j_1}$}; \node (v8) at (2,-1) {}; \node (v17) at (3,0) {}; \node
(v18) at (5,0) {}; \node (v9) at (6,-1) {}; \node (v10) at (7,-2)
{$e_{j_{t-1}}$}; \node (v11) at (8,-1) {}; \node (v13) at (9,0) {};
\node (v12) at (10,1) {$e_{i_t}$}; \draw [->] (v1) edge (v2); \draw
[->] (v3) edge (v4); \draw [->] (v1) edge (v5); \draw [->] (v6) edge
(v7); \draw [->] (v8) edge (v7); \draw [->] (v9) edge (v10); \draw
[->] (v11) edge (v10); \draw [->] (v12) edge (v13); \draw [dashed]
(v2) edge (v3); \draw [dashed] (v5) edge (v6); \draw [dashed] (v17)
edge (v18); \draw [dashed] (v13) edge (v11);
\end{tikzpicture}
\end{center}
$t=1,2,\cdots$ then $N$ is of the form
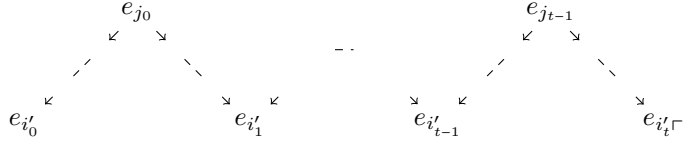
\begin{center}
\begin{tikzpicture}[xscale=0.5, yscale=0.5]
\node (v1) at (-3,2) {$e_{j_0}$}; \node (v2) at (-4,1) {}; \node
(v3) at (-5,0) {}; \node (v4) at (-6,-1) {$e_{i'_0}$}; \node (v5) at
(-2,1) {}; \node (v6) at (-1,0) {}; \node (v7) at (0,-1)
{$e_{i'_1}$}; \node (v8) at (1,0) {}; \node (v23) at (2,1) {}; \node
(v24) at (3,1) {}; \node (v15) at (4,0) {}; \node (v16) at (5,-1)
{$e_{i'_{t-1}}$}; \node (v17) at (6,0) {}; \node (v19) at (7,1) {};
\node (v18) at (8,2) {$e_{j_{t-1}}$}; \node (v20) at (9,1) {}; \node
(v21) at (10,0) {}; \node (v22) at (11,-1)
{$e_{i'_{t}\lefthalfcap}$}; \draw [->] (v1) edge (v2); \draw [->]
(v3) edge (v4); \draw [->] (v1) edge (v5); \draw [->] (v6) edge
(v7); \draw [->] (v8) edge (v7); \draw [->] (v15) edge (v16); \draw
[->] (v17) edge (v16); \draw [->] (v18) edge (v19); \draw [->] (v18)
edge (v20); \draw [->] (v21) edge (v22); \draw [dashed] (v20) edge
(v21); \draw [dashed] (v19) edge (v17); \draw [dashed] (v23) edge
(v24); \draw [dashed] (v2) edge (v3); \draw [dashed] (v5) edge (v6);
\end{tikzpicture}
\end{center}
$e_{j_{0}} \rightarrow \cdots  \rightarrow e_{i'_{0}}$ and
$e_{i_{t}} \rightarrow \cdots \rightarrow e_{j_{t-1}} \rightarrow \cdots \rightarrow
e_{i'_{t}\lefthalfcap}$ are maximal directed strings, $e_{i'_k} \leftarrow
\cdots \leftarrow e_{j_{k-1}}  \leftarrow \cdots \leftarrow
e_{i_k}=c_ke_{i'_k} \leftarrow \cdots \leftarrow e_{j_{k}}  \leftarrow
\cdots \leftarrow e_{i_k}$ in $A$ for $k=1,2,\cdots, t-1 $ and $c_k \in \kk^*$.
\end{lem}

The canonical map from $N$ to $M$ sends $e_{j_k}$ from the top of
$N$ to $d_ke_{j_k}$ in the socle of $M$ ($d_k \in \kk$) with all $d_k$ but one
equal to $0$. In the stable category all these maps belong to the same one-dimensional subspace of $\Homs(N,M)$.

\begin{lem}\label{one in, one out}
Let $Q$ be a quiver of a selfinjective special biserial algebra, and
let $x$ be a vertex of $Q$. There is only one arrow entering $x$ if
and only if there is only one arrow leaving $x$.
\end{lem}

\begin{proof}
If there are no arrows entering vertex $x$, then the simple module corresponding to $x$ is injective, and hence, it is projective and there are no arrows leaving $x$, the case with no arrows leaving $x$ is similar. Assume there is one arrow $\alpha$ entering some vertex and two
arrows $\beta, \gamma$ leaving it. Then either $\alpha\beta=0$ or
$\alpha\gamma=0$, say $\alpha\beta=0$. Then $\beta \in \soc(A)$,
hence $\beta$ is equal to some path starting from $\gamma$, which
can not happen, since the ideal of relations is admissible. The case of one arrow leaving the vertex and two arrow entering is similar.
\end{proof}

\begin{lem}\label{Ext}
Let $A$ be a selfinjective special biserial algebra and
let $\mathcal{M}$ be a maximal system of orthogonal stable bricks in
$\smod A$. For $M \in \mathcal{M}$, $\dim\Homs(\tau^{-1}M, \oplus_{M_i \in
\mathcal{M}}M_i) \leq 2$ and $\dim\Homs(\oplus_{M_i \in
\mathcal{M}}\tau^{-1}M_i,M) \leq 2$.
\end{lem}

\begin{proof}
We will prove only $\dim\Homs(\tau^{-1}M, \oplus_{M_i \in
\mathcal{M}}M_i) \leq 2$, the other statement follows from the
duality. Indeed, $A$ is selfinjective special biserial if and only
if $A^{op}$ is selfinjective special biserial. There is a duality
$D: \smod A \rightarrow \smod A^{op}$, which sends $\tau_A$ to
$\tau^{-1}_{A^{op}}$ and maximal systems of orthogonal stable bricks
in $\smod A$ to maximal systems of orthogonal stable bricks in
$\smod A^{op}$. Hence, if we prove $\dim\Homs(\tau^{-1}M, \oplus_{M_i
\in \mathcal{M}}M_i) \leq 2$ for any maximal system of orthogonal
stable bricks in $\smod A^{op}$, then $\dim\Homs(\oplus_{M_i \in
\mathcal{M}}\tau^{-1}M_i,M) \leq 2$ holds for any maximal system of
orthogonal stable bricks in $\smod A$.

Let $M \in \mathcal{M}$ be a module of the form
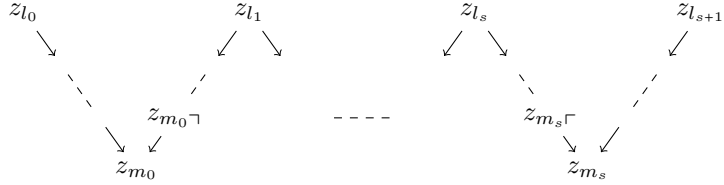
\begin{center}
\begin{tikzpicture}[xscale=0.5, yscale=0.7]
\node (v11) at (-8,3) {$z_{l_0}$}; \node (v1) at (-7,2) {}; \node
(v2) at (-6,1) {}; \node (v12) at (-5,0) {$z_{m_0}$}; \node (v4) at
(-4,1) {$z_{m_0\righthalfcap}$}; \node (v3) at (-3,2) {}; \node
(v13) at (-2,3) {$z_{l_1}$}; \node (v14) at (-1,2) {}; \node (v5) at
(0,1) {}; \node (v6) at (2,1) {}; \node (v16) at (3,2) {}; \node
(v15) at (4,3) {$z_{l_s}$}; \node (v7) at (5,2) {}; \node (v8) at
(6,1) {$z_{m_s\lefthalfcap}$}; \node (v17) at (7,0) {$z_{m_s}$};
\node (v10) at (8,1) {}; \node (v9) at (9,2) {}; \node (v18) at
(10,3) {$z_{l_{s+1}}$}; \draw [dashed] (v1) edge (v2); \draw
[dashed] (v3) edge (v4); \draw [dashed] (v5) edge (v6); \draw
[dashed] (v7) edge (v8); \draw [dashed] (v9) edge (v10); \draw [->]
(v11) edge (v1); \draw [->] (v2) edge (v12); \draw [->] (v4) edge
(v12); \draw [->] (v13) edge (v3); \draw [->] (v13) edge (v14);
\draw [->] (v15) edge (v16); \draw [->] (v15) edge (v7); \draw [->]
(v8) edge (v17); \draw [->] (v10) edge (v17); \draw [->] (v18) edge
(v9);
\end{tikzpicture}
\end{center}
where the first or the last directed substring may be trivial. The diagram of $\tau^{-1}M$ is
formed by adding hooks $z_{m_{-1}} \leftarrow \cdots \leftarrow
z_{l_0 \lefthalfcap} \rightarrow z_{l_0}$ and $z_{l_{s+1}}
\leftarrow z_{l_{s+1}\righthalfcap} \rightarrow \cdots \rightarrow
z_{m_{s+1}}$ (case $\rn 1$) or by deleting co-hooks $z_{l_0}
\rightarrow \cdots  \rightarrow z_{m_0} \leftarrow z_{m_0
\righthalfcap}$ and $z_{m_{s}\lefthalfcap} \rightarrow z_{m_{s}}
\leftarrow \cdots \leftarrow z_{l_{s+1}}$ (case $\rn 2$) or by
adding a hook $z_{m_{-1}} \leftarrow \cdots \leftarrow z_{l_0
\lefthalfcap} \rightarrow z_{l_0}$ and deleting a co-hook
$z_{m_{s}\lefthalfcap} \rightarrow z_{m_{s}} \leftarrow \cdots
\leftarrow z_{l_{s+1}}$ (case $\rn 3$). Note that after deleting a co-hook of the form $z_{l_0}
\rightarrow \cdots  \rightarrow z_{m_0} \leftarrow z_{m_0
\righthalfcap}$ the vertex $z_{m_0
\righthalfcap}$ stays intact. If $M \simeq \rad P$, then
$\tau^{-1}M \simeq P/\soc P$ (case $\rn 4$).

We are going to use the same notation for morphisms in $\mod A$ and
the corresponding morphisms in $\smod A$. There are canonical
diagram morphisms $M \rightarrow \tau^{-1} M$ induced by the intersection of diagrams. In case ($\rn 1$)
there is a monomorphism $f : M \rightarrow \tau^{-1} M$, in case
($\rn 2$) there is an epimorphism $f : M \rightarrow \tau^{-1} M$,
in case ($\rn 3$) there is a composition of a monomorphism and
an epimorphism $f : M \rightarrow \tau^{-1} M$. The map $f$ is 
equal to zero in the stable category iff in case ($\rn 3$) module $M$ is a maximal directed string $ z_{l_0} \rightarrow \cdots
\rightarrow z_{m_0}$  (case $\rn 3$'). Note that in this last case $M$ can be a simple module
corresponding to a vertex with one incoming and one outgoing arrow. In case ($\rn 4$) there are two morphisms $f$ and
$f'$, with images equal to two indecomposable summands of $\rad P/\soc P$ (if $P$ is not uniserial), $f=cf'\neq 0$ ($c\in\kk^*$)
in $\smod A$. If $P$ is uniserial, then $f=0$.

If there is a morphism $g: \tau^{-1} M \rightarrow M_i$, then it
factors through $\Cone(f)$, since $gf=0$ in $\smod A$ by the
definition of the maximal system of orthogonal stable bricks, even
if $M_i \simeq M$. Here $\Cone(f)$ denotes the cone of a morphism $f$ in the triangulated structure on $\smod A$. Let us compute $\Cone(f)$. 

case ($\rn 1$): Since $f$ is a monomorphism, $\Cone(f) \simeq \Coker(f)
\simeq z_{m_{-1}} \leftarrow \cdots \leftarrow z_{l_0 \lefthalfcap}
\oplus z_{l_{s+1}\righthalfcap} \rightarrow \cdots \rightarrow
z_{m_{s+1}}$ is a sum of two maximal directed strings. (If the hook was trivial, then this is just a simple
module.)

case ($\rn 2$): Since $f$ is an epimorphism, $\Cone(f) \simeq
\Omega^{-1}\Ker(f) \simeq \Omega^{-1}(z_{l_0} \rightarrow \cdots
\rightarrow z_{m_0} \oplus z_{m_{s}} \leftarrow \cdots \leftarrow
z_{l_{s+1}}) \simeq z_{m_0 \righthalfcap} \leftarrow \cdots
\leftarrow z_{l_1} \leftarrow \cdots \leftarrow z_{l_0 \lefthalfcap}
\oplus z_{l_{s+1} \righthalfcap} \rightarrow \cdots \rightarrow
z_{l_s} \rightarrow \cdots \rightarrow z_{m_{s} \lefthalfcap}$ is a
sum of two maximal directed strings (in the case, where $z_{m_0}$ corresponds
to a vertex with one incoming and one outgoing arrow and the co-hook
is trivial there still is a maximal directed string ending at $z_{m_0
\righthalfcap}$ and we are going to use the notation $z_{m_0
\righthalfcap} \leftarrow \cdots \leftarrow z_{l_1} \leftarrow
\cdots \leftarrow z_{l_0 \lefthalfcap}$ for it).

case ($\rn 3$): The morphism $f$ is a composition of a monomorphism and an
epimorphism, $\Cone(f)$ can be easily computed by the octahedron axiom
or by the definition of triangles in $\smod A$. As before, $\Cone(f)  \simeq
z_{m_{-1}} \leftarrow \cdots \leftarrow z_{l_0 \lefthalfcap} \oplus
z_{l_{s+1} \righthalfcap} \rightarrow \cdots \rightarrow z_{l_s}
\rightarrow \cdots \rightarrow z_{m_{s} \lefthalfcap}$ is a sum of two maximal directed strings. In case ($\rn 3$') let $M$ be of the form $z_{m_{s}}=z_{l_0} \leftarrow \cdots
\leftarrow z_{l_{s+1}}$, then $\Cone(f)= \tau^{-1} M \oplus \Omega^{-1}(M)\simeq z_{m_{-1}} \leftarrow \cdots \leftarrow z_{l_0 \lefthalfcap} \oplus z_{l_{s+1} \righthalfcap} \rightarrow \cdots \rightarrow z_{m_{s} \lefthalfcap}=z_{l_0 \lefthalfcap}$.

case ($\rn 4$): In this case $s=0$. Assume that the projective
module $P$ is given by the relation $z_{l_0 \lefthalfcap}
\rightarrow z_{l_0} \rightarrow \cdots \rightarrow z_{m_0
\lefthalfcap} \rightarrow z_{m_0} = c z_{l_{s+1} \righthalfcap}
\rightarrow z_{l_{s+1}} \rightarrow \cdots \rightarrow z_{m_{s}
\righthalfcap} \rightarrow z_{m_s}$ ($c\in \kk^*$), where $z_{l_0
\lefthalfcap}=z_{l_{s+1} \righthalfcap}$ and $z_{m_0} = z_{m_s}$.
By the definition of triangles in
$\smod A$ we get $\Cone(f) \simeq z_{l_0 \lefthalfcap} \rightarrow z_{l_0}
\rightarrow \cdots \rightarrow z_{m_0 \lefthalfcap} \oplus
z_{l_{s+1} \righthalfcap} \rightarrow z_{l_{s+1}} \rightarrow \cdots
\rightarrow z_{m_{s}\righthalfcap}$ is again a sum of two maximal directed
strings. If $P$ is uniserial, $f=0$, then $\Cone(f)=P/\soc P\oplus \Omega^{-1}(\rad P)=P/\soc P \oplus \topp P$ is a sum of two maximal directed
strings, one of which is trivial.

Let $M_i$ be a module of the form (1), (2) or (3) from Lemma
\ref{s-proj}, assume there is a non-zero morphism $\tilde{g} :
\Cone(f) \rightarrow M_i$ in $\smod A$. Without loss of generality assume
there is a morphism $\tilde{g} :(z_{m_{-1}} \leftarrow \cdots
\leftarrow z_{l_0 \lefthalfcap}) \rightarrow M_i$. This morphism is
non-zero only in the following cases:
\begin{itemize}
\item case (1) $e_{j_{0}}=z_{l_0 \lefthalfcap}$ and the
composition of the last arrow in $e_{j_0} \leftarrow \cdots
\leftarrow e_{i_1}$ and the first arrow in $z_{m_{-1}} \leftarrow
\cdots \leftarrow z_{l_0 \lefthalfcap}$ is zero;
\item case (1) $e_{j_{t}}=z_{l_0
\lefthalfcap}$ and the composition of the last arrow in $e_{j_t}
\leftarrow \cdots \leftarrow e_{i_t}$ and the first arrow in
$z_{m_{-1}} \leftarrow \cdots \leftarrow z_{l_0 \lefthalfcap}$ is
zero;
\item case (2) $e_{i_{1}}=z_{l_0
\lefthalfcap}$ and $e_{i_1} \rightarrow \cdots \rightarrow e_{j_1}$
is a substring of $z_{m_{-1}} \leftarrow \cdots \leftarrow z_{l_0
\lefthalfcap}$;
\item case (2) $e_{i_{t}}=z_{l_0 \lefthalfcap}$ and $e_{i_t}
\rightarrow \cdots \rightarrow e_{j_{t-1}}$ is a substring of
$z_{m_{-1}} \leftarrow \cdots \leftarrow z_{l_0 \lefthalfcap}$;
\item case (3) $e_{j_{0}}=z_{l_0 \lefthalfcap}$
and the composition of the last arrow in $e_{j_0} \leftarrow \cdots
\leftarrow e_{i_1}$ and the first arrow in $z_{m_{-1}} \leftarrow
\cdots \leftarrow z_{l_0 \lefthalfcap}$ is zero;
\item case (3) $e_{i_{t}}=z_{l_0
\lefthalfcap}$ and $e_{i_t} \rightarrow \cdots \rightarrow
e_{j_{t-1}}$ is a substring of $z_{m_{-1}} \leftarrow \cdots
\leftarrow z_{l_0 \lefthalfcap}$.
\end{itemize}
Only one of all these cases can
occur, and for only one $M_i \in \mathcal{M}$, otherwise,
there would be a non-zero morphism between two objects from
$\mathcal{M}$, which is not identity in the case they coincide. With the same cases for the other maximal directed
string we get $\dim\Homs(\tau^{-1}M,\oplus_{M_i \in
\mathcal{M}}M_i) \leq 2$.
\end{proof}

\begin{rem}\label{nonzero}
We have seen that $\dim\Homs(\tau^{-1}M, \oplus_{M_i \in
\mathcal{M}}M_i) \leq 2$. Now we are going to list all the cases, where $g: \tau^{-1}M \rightarrow M_i\neq 0$ in $\smod A$, i.e. $\tilde{g}h\neq 0$, where $h: \tau^{-1}M\rightarrow \Cone(f)$. In the above notation:
\end{rem}

\begin{itemize}
\item For $M_i$ of the form (1) from Lemma \ref{s-proj} the map
 $g \neq 0$, if and only if one of the following holds (we will write out the condition only 
for one end of the diagram):

$\circ$ $M$ is of the form ($\rn 1$), $e_{j_{0}}=z_{l_0
\lefthalfcap}$ and the composition of the last arrow in $e_{j_0}
\leftarrow \cdots \leftarrow e_{i_1}$ and the first arrow in
$z_{m_{-1}} \leftarrow \cdots \leftarrow z_{l_0 \lefthalfcap}$ is
zero, additionally, the subdiagram of $\tau^{-1}M$ starting from $z_{l_0}$ and
coinciding with the subdiagram of $\Omega(M_i)$ starting from
$e_{j_0 \lefthalfcup}$ ends in a deep of $\tau^{-1}M$ which is not
a deep of $\Omega(M_i)$ or it ends on a peak of $\Omega(M_i)$ which
is not a peak of $\tau^{-1}M$. 

$\circ$ $M$ is of the form ($\rn 3$), the condition is the same as
in the previous case.

$\circ$ $M$ is of the form ($\rn 3$'), $e_{j_{0}}=z_{l_0
\lefthalfcap}$ and the composition of the last arrow in $e_{j_0}
\leftarrow \cdots \leftarrow e_{i_1}$ and the first arrow in
$z_{m_{-1}} \leftarrow \cdots \leftarrow z_{l_0 \lefthalfcap}$ is
zero.

\item For $M_i$ of the form (2) from Lemma \ref{s-proj} the map
 $g \neq 0$, if and only if one of the following holds (we will write out the condition only 
for one end of the diagram):

$\circ$ $M$ is of the form ($\rn 1$), $e_{i_{1}}=z_{l_0
\lefthalfcap}$ and $e_{i_1} \rightarrow \cdots \rightarrow e_{j_1}$
is a subdiagram of $z_{m_{-1}} \leftarrow \cdots \leftarrow z_{l_0
\lefthalfcap}$.

$\circ$ $M$ is of the form ($\rn 2$) $e_{i_{1}}=z_{l_0 \lefthalfcap}$ and $e_{i_1} \rightarrow
\cdots \rightarrow e_{j_1}$ is a subdiagram of $z_{m_{0}\righthalfcap}=z_{m_{-1}}
\leftarrow \cdots  z_{l_1}\cdots  \leftarrow z_{l_0 \lefthalfcap}$,
$z_{l_1}$ belongs to $e_{i_1} \rightarrow \cdots \rightarrow
e_{j_1}$, if $z_{l_1}=e_{j_1}$ the subdiagram of $\tau^{-1}M$
starting from $z_{l_1}$ and coinciding with the subdiagram of
$\Omega(M_i)$ starting from $e_{j_1}$ (going in the direction of $e_{j_2}$) ends in a deep of
$\tau^{-1}M$ which is not a deep of $\Omega(M_i)$ or it ends on a
peak of $\Omega(M_i)$ which is not a peak of $\tau^{-1}M$.

$\circ$ $M$ is of the form ($\rn 3$) Condition here coincides with the previous case.

$\circ$ $M$ is of the form ($\rn 3$') $e_{i_{1}}=z_{l_0 \lefthalfcap}$ and $e_{i_1} \rightarrow
\cdots \rightarrow e_{j_1}$ is a subdiagram of $z_{m_{-1}}
\leftarrow \cdots z_{l_1} \cdots  \leftarrow z_{l_0 \lefthalfcap}$.

$\circ$ $M$ is of the form ($\rn 4$) $e_{i_{1}}=z_{l_0 \lefthalfcap}$.

\item For $M_i$ of the form (3) from Lemma \ref{s-proj} condition for $g$ to be non-zero can
be easily obtained as a combination of previous cases.
\end{itemize}
In all other cases the composition is either zero or factors through
a projective module.

\begin{cor}\label{number of arrows}
Let $A$ be a selfinjective special biserial algebra, let $B$ be a  selfinjective algebra and let $F: \smod B \rightarrow \smod
A$ be an equivalence of categories. Then in the quiver of $B$ there are at most tow incoming and at most two
outgoing arrows at each vertex.
\end{cor}

\begin{proof}
Let $\mathcal{M} =
\{M_1,\cdots, M_n\}$ be the image of the set of simple $B$-modules under $F$. Let $S, S_i$ be simple $B$-modules sent to $M,M_i \in \mathcal{M}$. By Auslander
formula $\Ext^1(S,S_i)\simeq D\Homs(\tau^{-1}S_i,S)$ and
$\Ext^1(S_i,S)\simeq D\Homs(\tau^{-1}S,S_i)$, but
$\Homs(\tau^{-1}S_i,S) \simeq \Homs(\tau^{-1}M_i,M)$ and
$\Homs(\tau^{-1}S,S_i) \simeq \Homs(\tau^{-1}M,M_i)$. The number of
arrows from the vertex corresponding to $S$ to the vertex
corresponding to $S_i$ coincides with $\dim \Ext^1(S,S_i)$, thus by the
previous lemma there are at most tow incoming and at most two
outgoing arrows at the vertex corresponding to $S$.
\end{proof}

\begin{opr}
Let $N$ be an indecomposable s-projective module with respect to a
maximal system of orthogonal stable bricks $\mathcal{M}$. An
$A$-module $R$ is said to be the s-radical of $N$ (we denote $R$ by
\ssrad(N)) if the following conditions are satisfied:

(1) $R$ does not contain any projective direct summands.

(2) There is a projective $A$-module $P$ and a right minimal almost split morphism $R \oplus P \rightarrow
N$ in $\rmod A$, here $P$ may be zero.
\end{opr}

\begin{lem}[see Lemma 6.6 \cite{Pog}]\label{s-rad}
Let $A$ be a selfinjective special biserial algebra and
let $\mathcal{M}$ be a maximal system of orthogonal stable bricks in
$\smod A$. Let $M \in \mathcal{M}$ and let $N$ be an indecomposable
s-projective $A$-module such that $\sstop(N) \simeq M$. Then $\ssrad(N)
= R_1 \oplus R_2$, where $R_1, R_2$ are indecomposable, in the notation of Lemma \ref{s-proj}, $R_1$ and $R_2$ can be computed applying operations  $^l(-)$ and $(-)^r$ to the string corresponding to $N$:

case (1): $R_1$ and $R_2$ are of the form
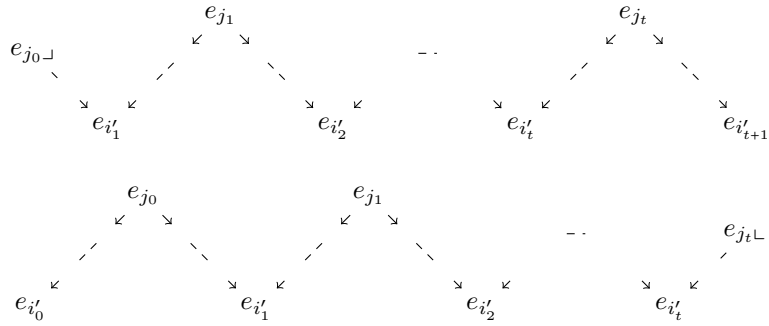
\begin{center}
\begin{tikzpicture}[xscale=0.5, yscale=0.5]
\node (v5) at (-2,1) {$e_{j_0\righthalfcup}$}; \node (v6) at (-1,0)
{}; \node (v7) at (0,-1) {$e_{i'_1}$}; \node (v8) at (1,0) {}; \node
(v10) at (2,1) {}; \node (v9) at (3,2) {$e_{j_1}$}; \node (v11) at
(4,1) {}; \node (v12) at (5,0) {}; \node (v13) at (6,-1)
{$e_{i'_2}$}; \node (v14) at (7,0) {}; \node (v23) at (8,1) {};
\node (v24) at (9,1) {}; \node (v15) at (10,0) {}; \node (v16) at
(11,-1) {$e_{i'_t}$}; \node (v17) at (12,0) {}; \node (v19) at
(13,1) {}; \node (v18) at (14,2) {$e_{j_t}$}; \node (v20) at (15,1)
{}; \node (v21) at (16,0) {}; \node (v22) at (17,-1)
{$e_{i'_{t+1}}$};
\draw [->] (v6) edge (v7); \draw [->] (v8) edge (v7); \draw [->]
(v9) edge (v10); \draw [->] (v9) edge (v11); \draw [->] (v12) edge
(v13); \draw [->] (v14) edge (v13); \draw [->] (v15) edge (v16);
\draw [->] (v17) edge (v16); \draw [->] (v18) edge (v19); \draw [->]
(v18) edge (v20); \draw [->] (v21) edge (v22); \draw [dashed] (v20)
edge (v21); \draw [dashed] (v19) edge (v17); \draw [dashed] (v23)
edge (v24);
\draw [dashed] (v5) edge (v6); \draw [dashed] (v10) edge (v8); \draw
[dashed] (v11) edge (v12);
\end{tikzpicture}
\end{center}

\begin{center}
\begin{tikzpicture}[xscale=0.5, yscale=0.5]
\node (v1) at (-3,2) {$e_{j_0}$}; \node (v2) at (-4,1) {}; \node
(v3) at (-5,0) {}; \node (v4) at (-6,-1) {$e_{i'_0}$}; \node (v5) at
(-2,1) {}; \node (v6) at (-1,0) {}; \node (v7) at (0,-1)
{$e_{i'_1}$}; \node (v8) at (1,0) {}; \node (v10) at (2,1) {}; \node
(v9) at (3,2) {$e_{j_1}$}; \node (v11) at (4,1) {}; \node (v12) at
(5,0) {}; \node (v13) at (6,-1) {$e_{i'_2}$}; \node (v14) at (7,0)
{}; \node (v23) at (8,1) {}; \node (v24) at (9,1) {}; \node (v15) at
(10,0) {}; \node (v16) at (11,-1) {$e_{i'_t}$}; \node (v17) at
(12,0) {}; \node (v19) at (13,1) {$e_{j_t\lefthalfcup}$};
\draw [->] (v1) edge (v2); \draw [->] (v3) edge (v4); \draw [->]
(v1) edge (v5); \draw [->] (v6) edge (v7); \draw [->] (v8) edge
(v7); \draw [->] (v9) edge (v10); \draw [->] (v9) edge (v11); \draw
[->] (v12) edge (v13); \draw [->] (v14) edge (v13); \draw [->] (v15)
edge (v16); \draw [->] (v17) edge (v16);
\draw [dashed] (v19) edge (v17); \draw [dashed] (v23) edge (v24);
\draw [dashed] (v2) edge (v3); \draw [dashed] (v5) edge (v6); \draw
[dashed] (v10) edge (v8); \draw [dashed] (v11) edge (v12);
\end{tikzpicture}
\end{center}

case (2): $R_1$ and $R_2$ are of the form
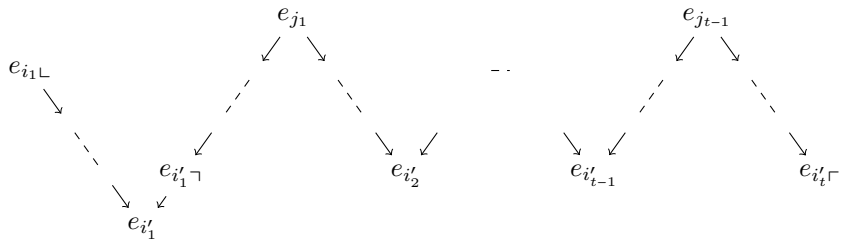
\begin{center}
\begin{tikzpicture}[xscale=0.5, yscale=0.7]
\node (v7) at (0,-1) {$e_{i'_1\righthalfcap} $}; \node (v8) at (1,0)
{}; \node (v10) at (2,1) {}; \node (v9) at (3,2) {$e_{j_1}$}; \node
(v11) at (4,1) {}; \node (v12) at (5,0) {}; \node (v13) at (6,-1)
{$e_{i'_2}$}; \node (v14) at (7,0) {}; \node (v23) at (8,1) {};
\node (v24) at (9,1) {}; \node (v15) at (10,0) {}; \node (v16) at
(11,-1) {$e_{i'_{t-1}}$}; \node (v17) at (12,0) {}; \node (v19) at
(13,1) {}; \node (v18) at (14,2) {$e_{j_{t-1}}$}; \node (v20) at
(15,1) {}; \node (v21) at (16,0) {}; \node (v22) at (17,-1)
{$e_{i'_{t} \lefthalfcap}$}; \draw [->] (v8) edge (v7); \draw [->]
(v9) edge (v10); \draw [->] (v9) edge (v11); \draw [->] (v12) edge
(v13); \draw [->] (v14) edge (v13); \draw [->] (v15) edge (v16);
\draw [->] (v17) edge (v16); \draw [->] (v18) edge (v19); \draw [->]
(v18) edge (v20); \draw [->] (v21) edge (v22); \draw [dashed] (v20)
edge (v21); \draw [dashed] (v19) edge (v17); \draw [dashed] (v23)
edge (v24); \draw [dashed] (v10) edge (v8); \draw [dashed] (v11)
edge (v12); \node (v1) at (-1,-2) {$e_{i'_1}$}; \node (v4) at
(-2,-1) {}; \node (v3) at (-3,0) {}; \node (v2) at (-4,1)
{$e_{i_1\lefthalfcup}$}; \draw [->] (v7) edge (v1); \draw [->] (v2)
edge (v3); \draw [->] (v4) edge (v1); \draw [dashed] (v3) edge (v4);
\end{tikzpicture}
\end{center}
where $e_{i_1} \rightarrow e_{i_1\lefthalfcup} \rightarrow \cdots
\rightarrow e_{i'_1} = c e_{i_1} \rightarrow \cdots \rightarrow e_{j_1}\rightarrow \cdots
\rightarrow e_{i'_1\righthalfcap} \rightarrow  e_{i'_1}$ in $A$ ($c\in \kk^*$)
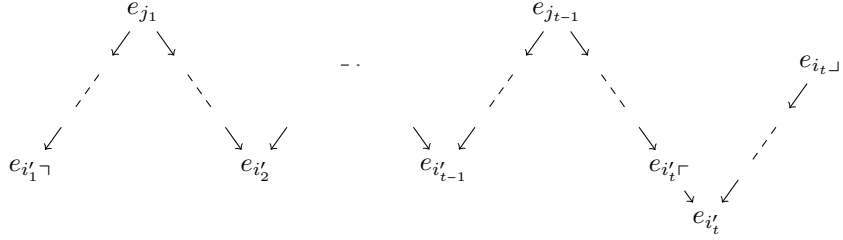
\begin{center}
\begin{tikzpicture}[xscale=0.5, yscale=0.7]
\node (v7) at (0,-1) {$e_{i'_1 \righthalfcap} $}; \node (v8) at
(1,0) {}; \node (v10) at (2,1) {}; \node (v9) at (3,2) {$e_{j_1}$};
\node (v11) at (4,1) {}; \node (v12) at (5,0) {}; \node (v13) at
(6,-1) {$e_{i'_2}$}; \node (v14) at (7,0) {}; \node (v23) at (8,1)
{}; \node (v24) at (9,1) {}; \node (v15) at (10,0) {}; \node (v16)
at (11,-1) {$e_{i'_{t-1}}$}; \node (v17) at (12,0) {}; \node (v19)
at (13,1) {}; \node (v18) at (14,2) {$e_{j_{t-1}}$}; \node (v20) at
(15,1) {}; \node (v21) at (16,0) {}; \node (v22) at (17,-1)
{$e_{i'_{t}\lefthalfcap}$}; \draw [->] (v8) edge (v7); \draw [->]
(v9) edge (v10); \draw [->] (v9) edge (v11); \draw [->] (v12) edge
(v13); \draw [->] (v14) edge (v13); \draw [->] (v15) edge (v16);
\draw [->] (v17) edge (v16); \draw [->] (v18) edge (v19); \draw [->]
(v18) edge (v20); \draw [->] (v21) edge (v22); \draw [dashed] (v20)
edge (v21); \draw [dashed] (v19) edge (v17); \draw [dashed] (v23)
edge (v24); \draw [dashed] (v10) edge (v8); \draw [dashed] (v11)
edge (v12); \node (v4) at (18,-2) {$e_{i'_t}$}; \node (v2) at
(19,-1) {}; \node (v1) at (20,0) {}; \node (v3) at (21,1)
{$e_{i_t\righthalfcup}$}; \draw [dashed] (v1) edge (v2); \draw [->]
(v3) edge (v1); \draw [->] (v2) edge (v4); \draw [->] (v22) edge
(v4);
\end{tikzpicture}
\end{center}
where $e_{i_t} \rightarrow \cdots \rightarrow e_{j_{t-1}}\rightarrow \cdots \rightarrow
e_{i'_t\lefthalfcap}\rightarrow e_{i'_t} =c e_{i_t} \rightarrow
e_{i_t\righthalfcup}\rightarrow \cdots \rightarrow e_{i'_t}$ in $A$ ($c\in \kk^*$)

case (3): $R_1$ and $R_2$ are of the form
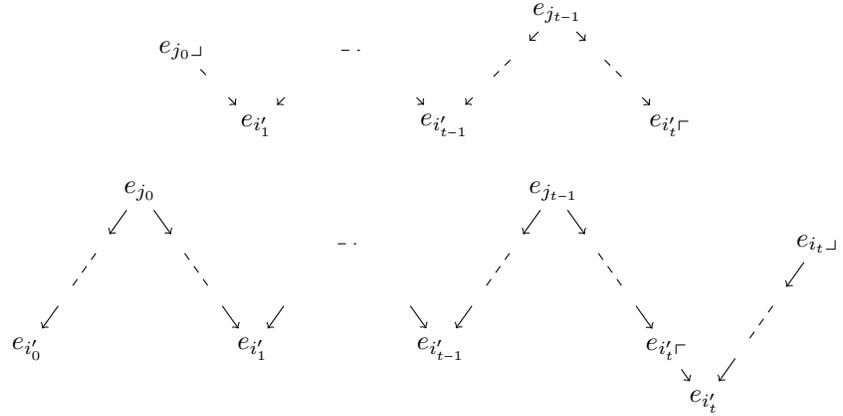
\begin{center}
\begin{tikzpicture}[xscale=0.5, yscale=0.5]
\node (v5) at (-2,1) {$e_{j_0 \righthalfcup}$}; \node (v6) at (-1,0)
{}; \node (v7) at (0,-1) {$e_{i'_1}$}; \node (v8) at (1,0) {}; \node
(v23) at (2,1) {}; \node (v24) at (3,1) {}; \node (v15) at (4,0) {};
\node (v16) at (5,-1) {$e_{i'_{t-1}}$}; \node (v17) at (6,0) {};
\node (v19) at (7,1) {}; \node (v18) at (8,2) {$e_{j_{t-1}}$}; \node
(v20) at (9,1) {}; \node (v21) at (10,0) {}; \node (v22) at (11,-1)
{$e_{i'_{t} \lefthalfcap}$};
\draw [->] (v6) edge (v7); \draw [->] (v8) edge (v7); \draw [->]
(v15) edge (v16); \draw [->] (v17) edge (v16); \draw [->] (v18) edge
(v19); \draw [->] (v18) edge (v20); \draw [->] (v21) edge (v22);
\draw [dashed] (v20) edge (v21); \draw [dashed] (v19) edge (v17);
\draw [dashed] (v23) edge (v24);
\draw [dashed] (v5) edge (v6);
\end{tikzpicture}
\end{center}
\begin{center}
\begin{tikzpicture}[xscale=0.5, yscale=0.7]
\node (v1) at (-3,2) {$e_{j_0}$}; \node (v2) at (-4,1) {}; \node
(v3) at (-5,0) {}; \node (v4) at (-6,-1) {$e_{i'_0}$}; \node (v5) at
(-2,1) {}; \node (v6) at (-1,0) {}; \node (v7) at (0,-1)
{$e_{i'_1}$}; \node (v8) at (1,0) {}; \node (v23) at (2,1) {}; \node
(v24) at (3,1) {}; \node (v15) at (4,0) {}; \node (v16) at (5,-1)
{$e_{i'_{t-1}}$}; \node (v17) at (6,0) {}; \node (v19) at (7,1) {};
\node (v18) at (8,2) {$e_{j_{t-1}}$}; \node (v20) at (9,1) {}; \node
(v21) at (10,0) {}; \node (v22) at (11,-1) {$e_{i'_{t}
\lefthalfcap}$}; \draw [->] (v1) edge (v2); \draw [->] (v3) edge
(v4); \draw [->] (v1) edge (v5); \draw [->] (v6) edge (v7); \draw
[->] (v8) edge (v7); \draw [->] (v15) edge (v16); \draw [->] (v17)
edge (v16); \draw [->] (v18) edge (v19); \draw [->] (v18) edge
(v20); \draw [->] (v21) edge (v22); \draw [dashed] (v20) edge (v21);
\draw [dashed] (v19) edge (v17); \draw [dashed] (v23) edge (v24);
\draw [dashed] (v2) edge (v3); \draw [dashed] (v5) edge (v6); \node
(v9) at (12,-2) {$e_{i'_t}$}; \node (v10) at (13,-1) {}; \node (v12)
at (14,0) {}; \node (v11) at (15,1) {$e_{i_t \righthalfcup}$}; \draw
[->] (v22) edge (v9); \draw [->] (v10) edge (v9); \draw [->] (v11)
edge (v12); \draw [dashed] (v12) edge (v10);
\end{tikzpicture}
\end{center}
where $e_{i_t} \rightarrow \cdots\rightarrow e_{j_{t-1}}\rightarrow \cdots \rightarrow
e_{i'_t\lefthalfcap}\rightarrow e_{i'_t} =c e_{i_t} \rightarrow
e_{i_t\righthalfcup}\rightarrow \cdots \rightarrow e_{i'_t}$ in $A$ ($c\in \kk^*$). Note
that $R_1$ may be zero.
\end{lem}

\begin{cor}\label{simple top}
Let $\smod B \rightarrow \smod A$ be an equivalence of categories, where $B$ is selfinjective and $A$ is selfinjective special biserial. Let $\mathcal{M} =
\{M_1,\cdots, M_n\}$ be the image of the set of simple $B$-modules and let
$\{N_1,\cdots, N_n\}$ be the image of the corresponding modules of
the form $P/\soc P$, where $P$ is indecomposable projective. Then $\ssrad(N_i)$ is the image of the module of
the form $\rad P/ \soc P$. Moreover, indecomposable summands of
$\ssrad(N_i)$ have simple $\sstop$, that is, if $\ssrad(N_i) = R_1
\oplus R_2$, where $R_1, R_2$ are indecomposable, then $\dim\Homs(R_j,\oplus_{M_i \in
\mathcal{M}}M_i)=1$ for non-zero $R_j$.
\end{cor}

\begin{proof}
Note that for an indecomposable projective module $P$, $\dim \topp(
\rad P/ \soc P)$ corresponds to the number of arrows going out of the
vertex corresponding to $P$. By Corollary \ref{number of arrows} there are at most two arrows going out of the vertex
corresponding to $P$; thus, if $\rad P/ \soc P$ has two non-zero
non-projective summands $\tilde{R}_1$ and $\tilde{R}_2$, then both
$R_1$ and $R_2$ are non-zero, hence
$\dim\Homs(\tilde{R}_j,\oplus_{S_i \in \mathcal{S}}S_i)=1$ and
$\dim\Homs(R_j,\oplus_{M_i \in \mathcal{M}}M_i)=1$.

Assume now that for some $N$ the module $\ssrad(N) = R$ is indecomposable. By
the description of the Auslander-Reiten triangles in $\smod A$, the diagram of $N$
is a maximal directed string (which may or may not coincide with $P/\soc P$
for a uniserial projective module $P$). Then, $M=\sstop(N)$ is a maximal directed string
or a simple module (in case $P/\soc P$). The $A^{op}$-module $DM$ is also a maximal directed
string or a simple module. By arguments analogous to the proof of Lemma
\ref{Ext} and Remark \ref{nonzero}, $\dim\Homs(\tau^{-1}DM, \oplus_{M_i \in \mathcal{M}}DM_i)
= 1$. Note that, if $DM$ is simple corresponding to a vertex with one
incoming and one outgoing arrow the result follows from Lemma \ref{one in, one out}. As
$\dim\Homs(\tau^{-1}DM, \oplus_{M_i \in \mathcal{M}}DM_i) = 1 =
\dim\Homs(\tau^{-1}\oplus_{M_i \in \mathcal{M}}M_i, M)$, there
is one arrow going out of the vertex corresponding to $P$ and
$\dim\Homs(R,\oplus_{M_i \in \mathcal{M}}M_i)=1$.
\end{proof}

In the notation of Corollary \ref{simple top}, for an indecomposable projective $B$-module $P$ with $\topp (P)=S$, the dimension of $\topp (\rad(P))$ corresponds to the dimension of $\Ext^1(S,\oplus S_i)$, where $\oplus S_i$ is the sum of representatives of iso-classes of simple $B$ modules. If $\Ext^1(S,S_i)\neq 0$, then $S_i$ is a summand of $\topp (\rad(P))$. Since $\Ext^1(S,S_i) \simeq \Homs(\tau^{-1}S_i,S)$, using the equivalence of stable categories, $\sstop(R)$ corresponds to $\Homs(\tau^{-1}M_i,M)$. The
following lemma follows easily from Remark \ref{nonzero}. (The roles
of $M$ and $M_i$ are switched.)

\begin{lem}[see Lemma 6.9 \cite{Pog}]\label{s-top}
Let $A$ be a selfinjective special biserial algebra and
let $\mathcal{M}$ be a maximal system of orthogonal stable bricks in
$\smod A$, assume additionally, that $\mathcal{M}$ is an image of the set of simple $B$-modules
under some stable equivalence, where $B$ is selfijective. Let $M \in \mathcal{M}$ be as in Lemma \ref{s-proj}. Moreover, let $N$ be an indecomposable
s-projective $A$-module such that $\sstop(N) \simeq M$, in the notation of Lemma \ref{s-rad}, $\ssrad(N)
= R_1 \oplus R_2$, then $\sstop(R_1)$ and $\sstop(R_2)$ are of the
following form:

case (1): $\sstop(R_1)$ is
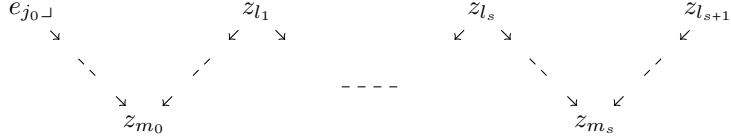
\begin{center}
\begin{tikzpicture}[xscale=0.5, yscale=0.5]
\node (v11) at (-8,3) {$e_{j_0\righthalfcup}$}; \node (v1) at (-7,2)
{}; \node (v2) at (-6,1) {}; \node (v12) at (-5,0) {$z_{m_0}$};
\node (v4) at (-4,1) {}; \node (v3) at (-3,2) {}; \node (v13) at
(-2,3) {$z_{l_1}$}; \node (v14) at (-1,2) {}; \node (v5) at (0,1)
{}; \node (v6) at (2,1) {}; \node (v16) at (3,2) {}; \node (v15) at
(4,3) {$z_{l_s}$}; \node (v7) at (5,2) {}; \node (v8) at (6,1) {};
\node (v17) at (7,0) {$z_{m_s}$}; \node (v10) at (8,1) {}; \node
(v9) at (9,2) {}; \node (v18) at (10,3) {$z_{l_{s+1}}$}; \draw
[dashed] (v1) edge (v2); \draw [dashed] (v3) edge (v4); \draw
[dashed] (v5) edge (v6); \draw [dashed] (v7) edge (v8); \draw
[dashed] (v9) edge (v10); \draw [->] (v11) edge (v1); \draw [->]
(v2) edge (v12); \draw [->] (v4) edge (v12); \draw [->] (v13) edge
(v3); \draw [->] (v13) edge (v14); \draw [->] (v15) edge (v16);
\draw [->] (v15) edge (v7); \draw [->] (v8) edge (v17); \draw [->]
(v10) edge (v17); \draw [->] (v18) edge (v9);
\end{tikzpicture}
\end{center}
where either the diagrams of $\sstop(R_1)$ and $R_1$ coincide or the subdiagram of $\sstop(R_1)$ starting from
$e_{j_0\righthalfcup}$ and coinciding with the subdiagram of $R_1$
starting from $e_{j_0\righthalfcup}$ ends on a deep of $\sstop(R_1)$
which is not a deep of $R_1$ or it ends on a peak of $R_1$ which is
not a peak of $\sstop(R_1)$ (note that this guarantees the existence
of a non-zero morphism from $R_1$ to $\sstop(R_1)$ which sends
$e_{j_0\righthalfcup}$ to $e_{j_0\righthalfcup}$ and is non-zero in
the stable category. Note also that this intersection can consist of
one vertex.);

$\sstop(R_2)$ is
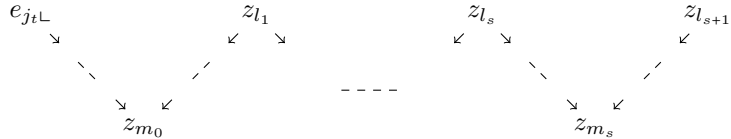
\begin{center}
\begin{tikzpicture}[xscale=0.5, yscale=0.5]
\node (v11) at (-8,3) {$e_{j_t\lefthalfcup}$}; \node (v1) at (-7,2)
{}; \node (v2) at (-6,1) {}; \node (v12) at (-5,0) {$z_{m_0}$};
\node (v4) at (-4,1) {}; \node (v3) at (-3,2) {}; \node (v13) at
(-2,3) {$z_{l_1}$}; \node (v14) at (-1,2) {}; \node (v5) at (0,1)
{}; \node (v6) at (2,1) {}; \node (v16) at (3,2) {}; \node (v15) at
(4,3) {$z_{l_s}$}; \node (v7) at (5,2) {}; \node (v8) at (6,1) {};
\node (v17) at (7,0) {$z_{m_s}$}; \node (v10) at (8,1) {}; \node
(v9) at (9,2) {}; \node (v18) at (10,3) {$z_{l_{s+1}}$}; \draw
[dashed] (v1) edge (v2); \draw [dashed] (v3) edge (v4); \draw
[dashed] (v5) edge (v6); \draw [dashed] (v7) edge (v8); \draw
[dashed] (v9) edge (v10); \draw [->] (v11) edge (v1); \draw [->]
(v2) edge (v12); \draw [->] (v4) edge (v12); \draw [->] (v13) edge
(v3); \draw [->] (v13) edge (v14); \draw [->] (v15) edge (v16);
\draw [->] (v15) edge (v7); \draw [->] (v8) edge (v17); \draw [->]
(v10) edge (v17); \draw [->] (v18) edge (v9);
\end{tikzpicture}
\end{center}
where either the diagrams of $\sstop(R_2)$ and $R_2$ coincide or the subdiagram of $\sstop(R_2)$ starting from
$e_{j_t\lefthalfcup}$ and coinciding with the subdiagram of $R_2$
starting from $e_{j_t\lefthalfcup}$ ends on a deep of $\sstop(R_2)$
which is not a deep of $R_2$ or it ends on a peak of $R_2$ which is
not a peak of $\sstop(R_2)$;

case (2): $\sstop(R_1)$ is
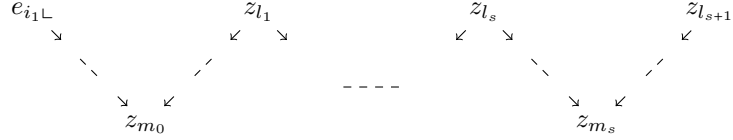
\begin{center}
\begin{tikzpicture}[xscale=0.5, yscale=0.5]
\node (v11) at (-8,3) {$e_{i_1\lefthalfcup}$}; \node (v1) at (-7,2)
{}; \node (v2) at (-6,1) {}; \node (v12) at (-5,0) {$z_{m_0}$};
\node (v4) at (-4,1) {}; \node (v3) at (-3,2) {}; \node (v13) at
(-2,3) {$z_{l_1}$}; \node (v14) at (-1,2) {}; \node (v5) at (0,1)
{}; \node (v6) at (2,1) {}; \node (v16) at (3,2) {}; \node (v15) at
(4,3) {$z_{l_s}$}; \node (v7) at (5,2) {}; \node (v8) at (6,1) {};
\node (v17) at (7,0) {$z_{m_s}$}; \node (v10) at (8,1) {}; \node
(v9) at (9,2) {}; \node (v18) at (10,3) {$z_{l_{s+1}}$}; \draw
[dashed] (v1) edge (v2); \draw [dashed] (v3) edge (v4); \draw
[dashed] (v5) edge (v6); \draw [dashed] (v7) edge (v8); \draw
[dashed] (v9) edge (v10); \draw [->] (v11) edge (v1); \draw [->]
(v2) edge (v12); \draw [->] (v4) edge (v12); \draw [->] (v13) edge
(v3); \draw [->] (v13) edge (v14); \draw [->] (v15) edge (v16);
\draw [->] (v15) edge (v7); \draw [->] (v8) edge (v17); \draw [->]
(v10) edge (v17); \draw [->] (v18) edge (v9);
\end{tikzpicture}
\end{center}
where either the diagrams of $\sstop(R_1)$ and $R_1$ coincide or the subdiagram of $\sstop(R_1)$ starting from
$e_{i_1\lefthalfcup}$ and coinciding with the subdiagram of $R_1$
starting from $e_{i_1\lefthalfcup}$ ends on a deep of $\sstop(R_1)$,
which is not a deep of $R_1$ or it ends on a peak of $R_1$ which is
not a peak of $\sstop(R_1)$;

$\sstop(R_2)$ is
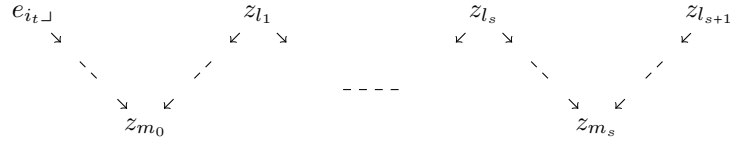
\begin{center}
\begin{tikzpicture}[xscale=0.5, yscale=0.5]
\node (v11) at (-8,3) {$e_{i_t\righthalfcup}$}; \node (v1) at (-7,2)
{}; \node (v2) at (-6,1) {}; \node (v12) at (-5,0) {$z_{m_0}$};
\node (v4) at (-4,1) {}; \node (v3) at (-3,2) {}; \node (v13) at
(-2,3) {$z_{l_1}$}; \node (v14) at (-1,2) {}; \node (v5) at (0,1)
{}; \node (v6) at (2,1) {}; \node (v16) at (3,2) {}; \node (v15) at
(4,3) {$z_{l_s}$}; \node (v7) at (5,2) {}; \node (v8) at (6,1) {};
\node (v17) at (7,0) {$z_{m_s}$}; \node (v10) at (8,1) {}; \node
(v9) at (9,2) {}; \node (v18) at (10,3) {$z_{l_{s+1}}$}; \draw
[dashed] (v1) edge (v2); \draw [dashed] (v3) edge (v4); \draw
[dashed] (v5) edge (v6); \draw [dashed] (v7) edge (v8); \draw
[dashed] (v9) edge (v10); \draw [->] (v11) edge (v1); \draw [->]
(v2) edge (v12); \draw [->] (v4) edge (v12); \draw [->] (v13) edge
(v3); \draw [->] (v13) edge (v14); \draw [->] (v15) edge (v16);
\draw [->] (v15) edge (v7); \draw [->] (v8) edge (v17); \draw [->]
(v10) edge (v17); \draw [->] (v18) edge (v9);
\end{tikzpicture}
\end{center}
where either the diagrams of $\sstop(R_2)$ and $R_2$ coincide or the subdiagram of $\sstop(R_2)$ starting from
$e_{i_t\righthalfcup}$ and coinciding with the subdiagram of $R_2$
starting from $e_{i_t\righthalfcup}$ ends on a deep of $\sstop(R_2)$
which is not a deep of $R_2$ or it ends on a peak of $R_2$ which is
not a peak of $\sstop(R_2)$;

case (3) is analogous to case (1)-$R_1$ for $R_1$ and case (2)-$R_2$ for
$R_2$.

\end{lem}

We are going to use the following criterion to prove that a
selfinjective algebra stably equivalent to a special biserial
algebra is stably biserial. Here we cite only the part of the result that we need. Note that this proposition was reproved
in \cite{AIP}:

\begin{prop}[Proposition 2.7 \cite{Pog}, Proposition 7.8 \cite{AIP}]\label{crit}
If a selfinjective algebra $B$ satisfies the following
conditions, then $B$ is Morita equivalent to an algebra, that satisfies conditions (a) and (c) from Definition \ref{stb}.

(a) For each indecomposable projective module $P$, we have
$\rad(P)/\soc(P) = X' \oplus X''$, (where $X' \neq 0$) such that
$\topp(X')$, $\topp(X'')$, $\soc(X')$, $\soc(X'')$ are simple
modules (or zero, in case $X''$ is zero).

(b) Let $X = X'$ or $X''$, and let $Q$ be the projective cover of
$X$. Then $X$ is non-projective and we denote by $p$ the epimorphism
$Q/\soc(Q) \rightarrow X$. Suppose that $\rad(Q)/\soc(Q) = Y_1
\oplus Y_2$, where $Y_1$ and $Y_2$ are indecomposable modules. Then,
for irreducible morphisms $w_1 : Y_1 \rightarrow Q/\soc(Q)$, $w_2 :
Y_2 \rightarrow Q/\soc(Q)$, $pw_1$ or $pw_2$ factors through a
projective module.

\end{prop}

To use the criterion above we need the following lemma:

\begin{lem}[see Proposition 7.1 \cite{Pog}]\label{relations}
Let $A$ be selfinjective special biserial, let $\mathcal{M}$ be a maximal system of
orthogonal stable bricks which is an image of the set of simple
$B$-modules under some stable equivalence, where $B$ is selfinjective. Let $N$ be s-projective
and $M \in \mathcal{M}$ be $\sstop(N)$. For $\ssrad(N) = R_1 \oplus
R_2$, where $R_1, R_2$ are indecomposable, let $\sstop(R_i)=Y \in \mathcal{M}$ and let $Q$ be an
indecomposable s-projective such that $\sstop(Q)=Y$, let $L_1 \oplus
L_2$ be the s-radical of $Q$, where $L_1, L_2$ are indecomposable. There exist $\underline{f}: Q
\rightarrow R_i$ and $\underline{h}: R_i \rightarrow
Y$ with $\underline{hf} \neq 0$ such that for irreducible morphisms $g_1: L_1 \rightarrow Q$,
$g_2:L_2 \rightarrow Q$, we have
$\underline{fg_1}=0$ or $\underline{fg_2}=0$.
\end{lem}

\begin{proof}
By Lemma \ref{s-top} in all the cases $R_i$ and $\sstop(R_i)$ start from the same
vertex and their intersection ends on a deep of $\sstop(R_i)$ which
is not a deep of $R_i$, or it ends on a peak of $R_i$ which is not a
peak of $\sstop(R_i)$ or $R_i$ and $\sstop(R_i)$ coincide. That guarantees the existence of a morphism
from $R_i$ to $\sstop(R_i)$, which sends this intersection to itself
and this morphism is non-zero in $\smod A$, let us denote this morphism by $h$.

Without loss of generality we can consider the case (1)-$R_1$. In each
case $\sstop(R_1)$ is itself a module of the form (1)-(3) from Lemma
\ref{s-proj}.

If $\sstop(R_1)$ has the form (1), then there is a non-zero morphism $f$
from $Q$ to $R_1$, whose image consists only of $e_{j_0
\righthalfcup}$. There is a summand $L_1$ of
$\ssrad(Q)$ which is formed by deleting the hook starting with
$e_{j_0 \righthalfcup}$, clearly $fg_1=0$ and $\underline{hf} \neq 0$. If
$\sstop(R_1)$ has the form (2), then there is a non-zero morphism $f$ from
$Q$ to $R_1$, induced by $z_{m_0}\rightarrow e_{i'_1}$. There is a summand $L_1$ of $\ssrad(Q)$ which is formed by
adding a co-hook starting from $e_{j'_0 \righthalfcup}$, the
composition  $fg_1$ factors through the projective
module with the top corresponding to $e_{j_0 \righthalfcup}$, clearly $\underline{hf} \neq 0$. The
case when $\sstop(R_1)$ has the form (3) is similar.
\end{proof}

\begin{thm}[see Theorem 7.3 \cite{Pog}]\label{Thm}
Let $A$ be a selfinjective special biserial $\kk$-algebra not isomorphic to the Nakayama algebra with $\rad^2=0$. If $B$ is a
basic algebra stably equivalent to $A$, then $B$ is
stably biserial.
\end{thm}

\begin{proof}
Let $\Phi: \smod B \rightarrow \smod A$ be an
equivalence of categories. Since $A$ is a selfinjective special biserial $\kk$-algebra not isomorphic to the Nakayama algebra with $\rad^2=0$ we can assume that $B$ is selfinjective. Indeed, since over $A$ for any Auslander-Reiten sequence $0\rightarrow M \xrightarrow{f}N \oplus P\rightarrow L\rightarrow 0$, where $P$ is projective and $N$ is not projective, we have $\underline{f}\neq 0$, then by \cite[Proposition 2.3]{AR6} $0\rightarrow \Phi^{-1}(M) \rightarrow \Phi^{-1}(N) \oplus Q\rightarrow \Phi^{-1}(L)\rightarrow 0$ is the Auslander-Reiten sequence for some projective $Q$. Hence, $\tau$ and $\tau^{-1}$ are defined for all not projective modules, so $B$ is selfinjective.

Let $B$ be a selfinjective algebra which is not a local Nakayama
algebra. Then by
Proposition \ref{tau-period} none of the simple $B$-modules
$\{S_i\}_{i=1,...n}$ and none of the modules of the form $P/\soc P$
for an indecomposable projective $B$-module $P$ are of $\tau$-period $1$. Thus
$\{\Phi(S_i)\}_{i=1,...n}=\mathcal{M}$ is a maximal system of
orthogonal stable bricks over $A$. As $\{\Phi(P_i/\soc
P_i)\}_{i=1,...n}$ is the set of s-projective modules with respect
to $\mathcal{M}$, $\Phi$ sends $\rad(P_i/\soc P_i)$ to
$\ssrad(\Phi(P_i/\soc P_i))$. Corollary \ref{simple top} implies that
$\rad(P_i/\soc P_i)$ is a sum of at most two modules with simple
top.

The duality $D_B: \rmod B \rightarrow \rmod B^{op}$ sends simple
$B$-modules to simple $B^{op}$-modules, modules of the form
$\rad(P_i/\soc P_i)$ to modules of the form $\rad(P_i/\soc P_i)$,
top to socle and socle to top. The equivalence $\Phi$ induces an equivalence $\smod
B^{op} \rightarrow \smod A^{op}$. Since $A^{op}$ is also
selfinjective special biserial, $\rad(P_i/\soc P_i)$ is a sum of at
most two modules with simple top. Hence, the $B$-module $\rad(P_i)/\soc
(P_i)$ is a sum of at most two modules with simple socle. Thus, the
condition (a) of Proposition \ref{crit} holds. These conditions
correspond to the fact that there are at most two incoming and
outgoing arrows in the quiver of $B$.

In the notations of Proposition \ref{crit}, by Lemma \ref{relations}
there exists $p: Q/\soc Q \rightarrow X$ such that condition (b)
holds. Let us prove that condition (b) holds for any $p': Q/\soc Q
\rightarrow X$. Let us denote by $\pi_X: Q \rightarrow X$ the
projective cover of $X$ and by $\pi: Q \rightarrow Q/\soc Q$ the
projective cover of $Q/\soc Q$. By assumption $X$ has a simple top,
thus without loss of generality we can assume that the image of
$p''=p-p'$ belongs to $\rad (X)$. The morphism $p''$ can be lifted to a morphism
$\tilde{p}: Q \rightarrow Q$ between the projective covers ($\pi_X
\tilde{p}=p''\pi$). The image of $\tilde{p}$ belongs to $\rad (Q)$;
hence $\tilde{p}$ factors through $Q/\soc(Q)$ and
$\tilde{p}=h\pi$ for some $h$. Thus, $\pi_X h\pi=p''\pi$ and since
$\pi$ is an epimorphism $\pi_X h=p''$. We get that $p''$ factors
through a projective, and hence is zero in the stable category,
$\underline{p}=\underline{p'}$ and condition (b) of Proposition
\ref{crit} holds. This condition correspond to condition (c) in
Definition \ref{stb}.

It is clear  that the
conditions (b) and (c) of Definition \ref{stb} are dual to each
other. By the previous paragraph condition (b) of Proposition
\ref{crit} holds for $B^{op}$, and thus condition (c) in Definition
\ref{stb} holds for $B^{op}$; thus condition (b) in Definition
\ref{stb} holds for $B$ and $B$ is stably biserial.
\end{proof}

\section{Auslander-Reiten conjecture}

In this section we are going to prove the Auslander-Reiten
conjecture for special biserial algebras.

Let $B$ be a stably biserial algebra. It is clear that $B/\soc(B)$
is a string algebra, and hence the classification of indecomposable
non-projective modules over $B$ coincides with the usual
classification using string and band modules. Then by \cite[Proposition
4.5]{AR5} all Auslander-Reiten sequences over $B$ and
$B/\soc(B)$ not ending with a $B$-module of the form $P/\soc(P)$
coincide. Hence, if there is a system of orthogonal stable bricks
$\mathcal{M}$ over $B$, then all the modules in $\mathcal{M}$ are
string modules.

\begin{lem}[compare to Lemma 4.1 \cite{Pog}]
Let $A=\kk Q/I$ be a stably biserial algebra and let $\mathcal{M}=\{M_1, \dots, M_k\}$ be a
system of orthogonal stable bricks. Then every simple $A$-module can
appear in the multiset of endpoints of diagrams corresponding to $M_i \in \mathcal{M}$ at most twice.
\end{lem}

\begin{proof} Let us fix some $v\in Q_0$. We will consider the simple module corresponding to $v$ and diagrams of $M_i \in \mathcal{M}$ ending at $v$, that is $M_i=c_1\cdots c_l$, $s(M_i)=v$ or $e(M_i)=v$. Suppose that some arrow $\alpha$ incident to $v$ occurs twice at the endpoint $v$ of some diagrams $M_{i_1}=c_1\cdots c_l, M_{i_2}=d_1\cdots d_t$ for some $1 \leq i_1,i_2 \leq k$ in the same manner. Taking the opposite strings $M_{i_j}^{-1}$ if necessary, we can assume that either $s(M_{i_j})=v$, $\alpha=c_1=d_1$ or  $s(M_{i_j})=v$, $\alpha=c_1^{-1}=d_1^{-1}$.  In both cases, there is a non-zero morphism $f:M_{i_1}\to M_{i_2}$ or $f:M_{i_2}\to M_{i_1}$, corresponding to the common part of the diagrams $M_{i_1},  M_{i_2}$. The morphism $f$ is non-zero in $\smod B$, this is a contradiction to the definition of a system of orthogonal bricks.

Now we are to show that at most two different arrows, incident to $v$ can occur at the endpoint $v$ of the diagrams of $M_i \in \mathcal{M}$. If there is only one incoming or outgoing arrow at $v$ (and, consequently, only one outgoing or incoming arrow at $v$, see Lemma \ref{one in, one out}), there is nothing to prove. So suppose that there are $\alpha_1,\alpha_2, \beta_1,\beta_2$ with $s(\alpha_1)=s(\alpha_2)=e(\beta_1)=e(\beta_2)=v$ and consider two cases (if there are loops at the vertex $v$, some arrows may coincide): $\{\beta_i\alpha_j\}_{i,j=1,2}\not\subseteq \soc(A)$ and $\{\beta_i\alpha_j\}_{i,j=1,2} \subseteq \soc(A)$.

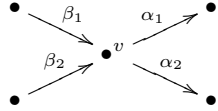
$$
\xymatrix @R=.5pc { \bullet \ar[rd]^{\beta_1}&&\bullet\\
& \bullet^v \ar[ru]^{\alpha_1} \ar[rd]^{\alpha_2}\\
\bullet \ar[ru]^{\beta_2}&&\bullet}
$$

{\it Case 1}.   Without  loss of generality we can assume $\beta_1\alpha_2\notin \soc(A)$. In this case, by stably biserial condition,
we have $\beta_1\alpha_1\in \soc(A),\beta_2\alpha_2\in \soc(A)$. Also in this case we have $\beta_2\alpha_1\neq 0$. Indeed, if $\beta_2\alpha_1= 0$, then $0\neq\beta_2 \alpha_2\in \soc(A)$ or $\beta_2\in \soc(A)$, which is impossible, hence if we consider a maximal path $q$ with $q\beta_1\alpha_2\neq 0$ ($q$ is of positive length, since $\beta_1\alpha_2\notin \soc(A)$), we have $\beta_2 \alpha_2=lq\beta_1 \alpha_2$ for some $l\in \kk^*$. As $q\beta_1\alpha_1\in q\cdot \soc(A)=0$, we have $\beta_2-lq\beta_1\in \soc(A)$, a contradiction, and thus $\beta_2\alpha_1\neq 0$.

Let us  prove that at least one of $\beta_1^{-1},\alpha_1$ does not occur at the endpoint of some $M_i \in \mathcal{M}$, and at least one of $\beta_2^{-1}, \alpha_2$ does not occur at the endpoint of some $M_i \in \mathcal{M}$ -- that is all we need. Take $j\in\{1,2\}$ and assume that both $\beta_j^{-1},\alpha_j$ occur at the endpoint of some $M, N \in \mathcal{M}$.

Let $M$ be a module with the diagram starting from $\alpha_i$ ($c_1=\alpha_i$), $x\in M$ -- an element corresponding to $v$, that is $xe_v=x,x\alpha_i\neq 0, x\alpha_{3-i}=0$, note that $x$ is non-zero in the top of $M$. Let $N$ be a module with diagram starting with $\beta_{i}^{-1}$ ($d_1=\beta_{i}^{-1}$), $y\in N$ is an  element corresponding to $v$. Note that $y$ belongs to the socle of $N$. Let $f:M\to N$ be the morphism with $f(x)=y$, which is zero in $\smod(A)$ by the definition of a system of orthogonal stable bricks. We claim that in this case $\tau N=N$ -- this also contradicts the definition of orthogonal stable bricks.
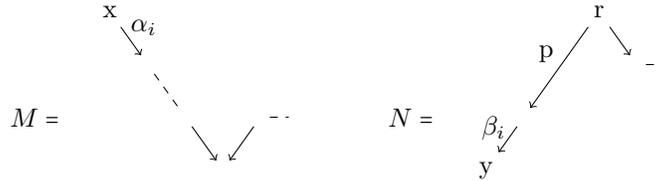
\begin{center}
\begin{tikzpicture}[xscale=0.5, yscale=0.7]

\node (v1) at (-5,4) {x};
\node (v2) at (-4,3) {};
\node (v3) at (-3,2) {};
\node (v4) at (-2,1) {};
\node (v5) at (-1,2) {};
\node (v6) at (0,2) {};
\node at (-7,2) {$M=$};
\node at (3,2) {$N=$};
\node (v8) at (5,1) {y};
\node (v7) at (6,2) {};
\node (v9) at (8,4) {r};
\node (v11) at (9,3) {};
\node (v12) at (10,3) {};
\draw [->] (v1) edge  (v2);
\draw [dashed] (v2) edge (v3);
\draw [->] (v3) edge (v4);
\draw [->] (v5) edge (v4);
\draw [dashed] (v5) edge (v6);
\draw [->] (v7) edge (v8);
\draw [->] (v9) edge (v11);
\draw [dashed] (v11) edge (v12);
\draw [->] (v9) edge (v7);
\node at (-4.1,3.7) {$\alpha_i$};
\node at (5.2,1.8) {$\beta_i$};
\node at (6.6,3.2) {p};
\end{tikzpicture}
\end{center}
We prove the latter claim by induction on the number of maximal directed substrings of $N$. Let $p\beta_i$, where $p$ is a path, correspond to the first maximal directed substring of $N$. Clearly $p\beta_i\notin \soc(A)$, as $N$ does not contain projective summands, and therefore $p\beta_i\alpha_s\neq 0$ for some $s$. We can assume that $s\neq i$. Indeed, if $s=i$, then $\beta_i\alpha_i\in \soc(A)$ implies $p=e_{s(\beta_i)}$ and in this case $p\beta_i\alpha_{3-i}\neq 0$ as well.

Let $t=s(p)$. The projective cover of $N$ is of the form $(g_1,g_2):P=P_t\oplus P'\to N$ where $g_1(e_t)=r$ is the  element of the basis corresponding to the first peak of $N$ (so we have $rp\beta_i=y$) and $y\notin \Imm(g_2)$. If $f=0\in \smod A$,
we have $f=gh=g_1h_1+g_2h_2$ for some $h=\big{(}\begin{smallmatrix}
         h_1 \\
         h_2
        \end{smallmatrix}\big{)} :M\to P_t\oplus P' $. As $g_1(p\beta_i)=y$ we can set $h(x)=(p\beta_i+z_1,z_2)$,
where $(z_1,z_2)\in \Ker(g)$. By construction of the projective cover, $z_1$ is a linear combination of paths not equal to $p\beta_i$ or subpaths of $p\beta_i$. Now $(0,0)=h(x\alpha_{3-i})=(p\beta_i\alpha_{3-i}+z_1\alpha_{3-i},z_2\alpha_{3-i})$, and therefore $0\neq p\beta_i\alpha_{3-i}=kp_1\alpha_{3-i}$ for some path $p_1\neq p\beta_i$ ($k \in \kk^*$). The case $p_1=p_1'\beta_i$ is impossible (in this case either both paths $p\beta_i\alpha_{3-i},p_1\alpha_{3-i}$ have lengths at least $3$ and contain subpaths of the form $\delta\gamma, \eta\gamma$ -- a contradiction, or $\beta_i\alpha_{3-i}$ is equal to a longer path ending with $\beta_i\alpha_{3-i}$, which is also impossible), therefore, as $\beta_{3-i}\alpha_{3-i}\in \soc(A)$, we have $p_1=\beta_{3-i}$. Note, that we get $p\beta_i\alpha_{3-i} \in \soc(A)$. Note that $p\neq \beta_{3-i} p_2$ for any path $p_2$ (else $p_1=\beta_{3-i}$ is a subpath of $p\beta_i$).

Now we can prove the base of our induction. The previous paragraph shows that $s(p)=s(\beta_{3-i})$. If $N$ is a directed string, corresponding to a maximal path $p\beta_i$ then $\tau^{-1}(N)$ is formed by adding a hook and deleting a co-hook, as $e(\beta_i)=e(\beta_{3-i})$, this hook is a maximal directed string, corresponding to $p\beta_i$. We see that $\tau^{-1}(N)=N$, as desired.

Note that we can compute $\tau^{-1}(N)$ in the usual way, since $N$ is not isomorphic to $\rad P$ for some projective module $P$.

Now suppose that the diagram of $N$ contains more than one maximal directed substrings. As $0=f(x\alpha_i)=g(p\beta_i\alpha_i+z_1\alpha_i,z_2\alpha_i)=g(z_1\alpha_i,z_2\alpha_i)$ we have $g_1(z_1\alpha_i)=0$ (since
$\Imm(g_1)\alpha_i\cap \Imm(g_2)\alpha_i=0$, as $\Imm(g_1)\cap \Imm(g_2)\in \soc (N)$), and, as $p\beta_i\alpha_{3-i}\in \soc(P_t)$, we have $g_1(p\beta_i\alpha_{3-i})=0$.  This implies that $g_1(\beta_{3-i}\alpha_i)=0$, $g_1(\beta_{3-i}\alpha_{3-i})=0$, since $\beta_{3-i}\alpha_{3-i} \in \soc (A)$, and hence the second maximal directed substring of the diagram of $N$ is an arrow $\beta_{3-i}$ ($g_1(\beta_{3-i})=g_1(z_1)\in \soc(N)$). Consider a module
$N'\leq N$, corresponding to the subdiagram, containing all but first two directed substrings of $N$ (deleting a hook of $N$). Then we have $\Imm(g_2)\subseteq N'$ and
$g_2h(x)=g_2(z_2)=-g_1(z_1)=lr\beta_{3-i}$ for some $l\in \kk^*$ (since  $0\neq g_1(\beta_{3-i})=g_1(z_1)$). This means that the module $N'$ and the morphism $f'=g_2h$ is of the same form as
$N$ and $f$ (in particular, $N'$ begins with $\beta_i^{-1}$ as well). By induction, the string corresponding to $N$ is of the form $\beta_i^{-1}p^{-1}\beta_{3-i}\beta_i^{-1}p^{-1}\beta_{3-i}\cdots\beta_i^{-1}p^{-1}$, and hence $N$ has $\tau$-period $1$.

{\it Case 2}.  $\{\beta_i\alpha_j\} \subseteq \soc(A)$. For each $i$, $\beta_i\notin \soc(A)$, so suppose that $\beta_i\alpha_{3-i}\neq 0$ (note that we can choose  different $j_1,j_2$ for $\beta_1,\beta_2$ with $\beta_1\alpha_{j_1}\neq 0$, $\beta_2\alpha_{j_2}\neq 0$, since in the other case we have $\beta_1\alpha_j=\beta_2\alpha_j=0$ for some $j$ and $\alpha_j\in \soc(A)$, which is impossible). Let us prove, as above (and with above notation)  that $\alpha_{j}$ and $\beta_i^{-1}$ cannot occur as first arrows for some $M$, $N$ by checking that the corresponding morphism $f$ is non-zero in $\smod A$.  As above, $f(x\alpha_{3-i})=0$ implies that there is a
path $p\neq\beta_i$ and $l\in \kk^*$ such that $\beta_i\alpha_{3-i}=lp\alpha_{3-i}$. As $\beta_i\alpha_{3-i},\beta_{3-i}\alpha_{3-i}\in \soc(A)$ we obtain that $p=\beta_{3-i}$ (otherwise a socle path would be a subpath of a longer path). This implies that $s(\beta_1)=s(\beta_2)$.

Now we have that all  directed strings containing $\beta_i$ has length $1$ and are maximal directed strings, and therefore $N$ is of the form
$\beta_i^{-1}\beta_{3-i}\beta_i^{-1}\beta_{3-i}\dots$. If the length of this word is odd, then $\tau(N)=N$ (deleting a co-hook and adding a hook does not change $N$), contradiction.
In the case of even length (i.e. if $dim(N)=2n+1$ is odd) let $y_1,\dots,y_n\in N$ be the elements of the diagram of $N$ corresponding to peaks. Then projective cover of $N$ is of the form $g:(e_{s(\beta_i)}A)^n\to N$, $g(z_k)=y_k$ for $k=1,\dots,n$, where $z_k$ is the generator of the corresponding copy of $e_{s(\beta_i)}A$ and $\Ker(g)=\langle \{z_k\beta_{3-i}-z_{k+1}\beta_i \} \rangle$. Now suppose that $f=gh$ for some $h$. Then
$h(x)=z_1\beta_i+\sum_{k=1}^{n-1}l_k(z_k\beta_{3-i}-z_{k+1}\beta_i)$. Multiplying this by $\alpha_{3-i}$, we obtain
$$
0=h(x\alpha_{3-i})=\sum_{k=1}^{n-1}z_k(l_k\beta_{3-i}\alpha_{3-i}-l_{k-1}\beta_{i}\alpha_{3-i})-l_{n-1}z_n\beta_{i}\alpha_{3-i},
$$
where $l_0=-1$. As all coefficients in the sum are to be zero, we obtain consequently that $l_i\neq 0$ for all $i=0,\dots n-1$,
therefore the last summand is non-zero, contradiction.
\end{proof}

Recall that a simple non-projective, non-injective module $S$ is called a node if the Auslander-Reiten sequence starting at $S$ has the form $$0\rightarrow S\rightarrow P \rightarrow \tau^{-1}S\rightarrow 0,$$ where $P$ is projective. By the results of \cite{MV2}, any algebra with nodes is stably equivalent to an algebra without nodes. Let $A$ be an algebra with nodes $S_1,\cdots,S_k$, $S=\oplus_{i=1}^kS_i$. Let $a$ be the trace of $S$ in $A$, i.e. $\Sigma_{h\in \Hom(S,A)} \Imm(h)$. Note that $a$ is a two-sided ideal of $A$. Let $b$ be a right annihilator of $a$, note that $A/b$ is semisimple and $a$ is an $A/a\text{-}A/b$ bimodule. Then the matrix algebra $\hat{T}_A=\begin{pmatrix}
A/a&a \\
0& A/b
\end{pmatrix}$ has no nodes and it is stably equivalent to $A$. The construction of $\hat{T}_A$ replaces every node in the quiver of $A$ by two simple modules: a sink and a source. It is clear, that the number of non-projective simple modules is preserved under this stable equivalence. 

\begin{thm}[compare to Theorem 0.1 \cite{Pog}]
Let $A,B$ be two finite dimensional algebras such that $\smod A\cong \smod B$ and $A$ is special biserial. Then the number of isomorphism classes of non-projective simple modules over $A$ and $B$ coincides.
\end{thm}
\begin{proof} Without loss of generality we can assume that $A, B$ have no semisimple summands. First, let us prove the statement for $A, B$ - selfinjective. If one of the algebras (and hence the other as well) has isolated vertices in the Auslander-Reiten quiver of the stable category, then they correspond to $P/\soc P$ or to $\rad P$ for some projective module $P$ of length $2$. Hence $A$ and $B$ have as summands Nakayama algebras with $\rad^2=0$, the number of simple modules over these algebras is the number of isolated vertices in the Auslander-Reiten quiver of the stable category, hence it is the same for $A, B$. From now on we can assume, that $A, B$ do not have a Nakayama algebra with $\rad^2=0$ as a summand. By Theorem \ref{Thm}, $B$ is stably biserial. Let $\mathcal{M}=\{M_1,\dots, M_k\}$ be the images of simple $A$-modules
under equivalence $F:\smod A\to \smod B$. Then $\mathcal{M}$ is a maximal system of orthogonal stable bricks. If some $M_i$ is a simple module, then it can not occur as an endpoint of any other diagram in $\mathcal{M}$.
The diagram of each non-simple $M_i$ has two endpoints, labelled by simple $B$-modules $S_i^1$ and $S_i^2$.  Suppose that the number of simple $B$-modules is less than $k$, then $S_{i_1}^{j_1}=S_{i_2}^{j_2}=S_{i_3}^{j_3}$ for some $i_l,j_l$. This contradicts the previous lemma. The same argument for the quasi-inverse $\tilde{F}:\smod B\to \smod A$ shows that the number of simple $B$-modules is less or equal to the number of simple $A$-modules and we are done.

Let us now consider arbitrary $A, B$, where $A$ is special biserial. If $A$ or $B$ has nodes, we can replace it by the matrix algebra $\hat{T}_A$ or $\hat{T}_B$, respectively. If $A$ is special biserial, then so is $\hat{T}_A$, so we can assume that $A, B$ have no nodes. To algebras $A, B$ one can associate selfinjective algebras $\Delta_{A},$ $\Delta_{B}$ in the following way: let $P_{A}$ be the set of isoclasses of projective-injective $A$-modules that remain projective-injective under the action of any power of the Nakayama functor $\nu^k$. Define $\Delta_{A}:=End(\oplus_{P\in P_{A}}P)$. If $A$ is special biserial, then $\Delta_{A}$ is selfinjective special biserial. By \cite{MV} (since $A, B$ have no nodes) the algebras $\Delta_{A},$ $\Delta_{B}$ are stably equivalent, and hence by the previous paragraph they have the same number of simple modules. By \cite{MV} $A, B$ have the same number of isomorphism classes of non-projective simple modules.
\end{proof}

\section{Symmetric stably biserial algebras}

Recall the standard description of a symmetric special biserial
algebra \cite{Sch}. We will assume that all quivers are  connected. Consider the following data:

\begin{enumerate}
\item A quiver $Q$ such that every vertex has two incoming and two outgoing arrows
or one incoming and one outgoing arrow.
\item A permutation $\pi$ on $Q_1$ with $e(\alpha)=s(\pi(\alpha))$ for all $\alpha\in Q_1$
\item A function $m:C(\pi)\to \mathbb{N}$, where $C(\pi)$ is the set of cycles of $\pi$.
\end{enumerate}

\noindent Now consider the ideal $I\subseteq \kk Q$ generated by the
following elements:
\begin{enumerate}

\item  $\alpha\beta$ for all $\alpha,\beta \in Q_1$, $\beta\neq\pi(\alpha)$
\item $\bigg(\alpha\pi(\alpha)\pi^2(\alpha)\dots\pi^{|
    \langle \pi \rangle \alpha|-1}(\alpha)\bigg)^{m(\langle \pi \rangle \alpha)}-\bigg(\beta\pi(\beta)\pi^2(\beta)\dots\pi^{{|\langle \pi \rangle \beta|-1}}(\beta)\bigg)^{m(\langle \pi \rangle \beta)}$ for all $\alpha,\beta\in Q_1$ with $s(\alpha)=s(\beta)$
\item $\bigg(\alpha\pi(\alpha)\pi^2(\alpha)\dots\pi^{|
    \langle \pi \rangle \alpha|-1}(\alpha)\bigg)^{m(\langle \pi \rangle
    \alpha)}\alpha$ and $\pi^{-1}(\alpha)\bigg(\alpha\pi(\alpha)\pi^2(\alpha)\dots\pi^{|
    \langle \pi \rangle \alpha|-1}(\alpha)\bigg)^{m(\langle \pi \rangle
    \alpha)}$ for all $\alpha \in Q_1$ such that $s(\alpha)$ has only one
    incoming and one outgoing arrow.

\end{enumerate}

Then $\kk Q/I$ is a symmetric special biserial algebra (SSB-algebra),
and each SSB-algebra can be described uniquely in this way, up to
obvious isomorphisms. Note that one of the relations from (3) is redundant.

The main aim of this section is to show that any symmetric stably
biserial algebra is in a sense a deformation of some
SSB-algebra. To obtain this, we are going to define the
permutation $\pi$ and the multiplicities of $\pi$-cycles for the
algebras from this class.

From now on let $A=\kk Q/I$ be an arbitrary stably biserial algebra,
with $I$ admissible. Let $\scl(A)=\soc(A)\setminus\{0\}$.

Case I. For $\alpha \in Q_1$ we put $\pi(\alpha)=\beta$ if
$\alpha\beta\notin \soc(A)$, $\beta\in Q_1$. The definition of a stably biserial algebra implies that we have at
most one such arrow.

If $\alpha\rad(A)\subseteq \soc (A)$ we are to define $\pi(\alpha)$
a bit more carefully.

Note that $\alpha\rad(A)= 0$ only for the case
$A=\kk[\alpha]/\alpha^2$ of the algebra with one vertex and one loop
$\alpha$, for that case $\pi(\alpha)=\alpha$, we are not going to
consider this case from here on. We can assume $\alpha\rad(A)\neq 0$
for any $\alpha\in Q_1$. Then (if $\alpha\rad(A)\subseteq \soc (A)$)
we have the following cases:

Case II. There  exist $\beta_1,\beta_2\in Q_1$ ($\beta_1\neq
\beta_2$) with $\alpha\beta_i\in \scl (A)$ ($i=1,2$). 

If $|Q_0|=1$ and $Q_1$ consists of two loops $\alpha,\beta$, then
$\alpha^2, \alpha\beta \in sc(A)$ implies $\beta\alpha \in sc(A)$.
If $\beta^2=0$ set $\pi(\alpha)=\beta, \pi(\beta)=\alpha$, if
$\beta^2 \in sc(A)$ we can chose $\pi(\alpha)=\alpha,
\pi(\beta)=\beta$. If $\beta^2 \notin \soc(A)$, set
$\pi(\alpha)=\alpha, \pi(\beta)=\beta$. From now on $|Q_0|>1$.

The arrow $\alpha$
isn't a loop -- otherwise $\beta_1,\beta_2$ are loops in the same
vertex and we have $|Q_0|=1$. Due to the symmetry, we have
$e(\beta_i)=s(\alpha), i=1,2$.

If $|Q_0|>2$ there exists a unique $\gamma\in Q_1$ with
$s(\gamma)=s(\alpha),e(\gamma)\neq e(\alpha)$ and there exists a
unique $\delta\in Q_1$ with $e(\delta)=e(\alpha),s(\delta)\neq
s(\alpha)$. Then we  have $\delta\beta_i\notin \soc(A)$ and
$\beta_i\gamma\notin \soc(A)$ for some $i$ and $\delta\beta_{3-i}=0$
and $\beta_{3-i}\gamma=0$ (as $\delta\beta_{3-i}$ and
$\beta_{3-i}\gamma$ belong to $\soc (A)$ by stably biserial condition and
are not cycles). Then $\pi(\delta)=\beta_i,
\pi(\beta_i)=\gamma$ as defined in Case I, and we can put
$\pi(\alpha)=\beta_{3-i},\pi(\beta_{3-i})=\alpha$.

Now consider the case $|Q_0|=2$. Due to the symmetry
$\beta_1\alpha,\beta_2\alpha\neq 0$ and clearly
$\beta_1\alpha,\beta_2\alpha\in sc(A)$,
$\beta_1\alpha=c\beta_2\alpha$,
$c\in \kk^*$. By symmetry
$\alpha\beta_1=c\alpha\beta_2$ as well. As $\beta_1-c\beta_2\notin
sc(A)$ (as a combination of non-closed paths), there exists
$\alpha_2\in Q_1$ with $\beta_1\alpha_2-c\beta_2\alpha_2\neq 0$. Then by stably biserial condition $\beta_i\alpha_2
\in \soc(A)$ for some $i$, and hence $\alpha_2\beta_i \in \soc(A)$ for the same $i$.
If $\beta_i\alpha_2 =0$, then $\alpha_2\beta_i =0$ and we can set
$\pi(\beta_i)=\alpha$, $\pi(\alpha)=\beta_i$,
$\pi(\beta_{3-i})=\alpha_2$, $\pi(\alpha_2)=\beta_{3-i}$ and
$\beta_{3-i}\alpha_2 \neq 0, \alpha_2\beta_{3-i} \neq 0$. If
$\beta_i\alpha_2 \neq 0$ but $\beta_{3-i}\alpha_2 \notin
\soc(A) $, then $\alpha_2\beta_{3-i} \notin \soc(A) $ and we
can set $\pi(\beta_i)=\alpha$, $\pi(\alpha)=\beta_i$,
$\pi(\beta_{3-i})=\alpha_2$, $\pi(\alpha)=\beta_{3-i}$ and
$\beta_{3-i}\alpha_2 \neq 0, \alpha_2\beta_{3-i} \neq 0$. If $\beta_{3-i}\alpha_2 \in sc(A)$,
$\beta_{i}\alpha_2\in sc(A)$, then we can chose $\pi$ arbitrary, e.g.
$\pi(\beta_i)=\alpha$, $\pi(\alpha)=\beta_i$,
$\pi(\beta_{3-i})=\alpha_2$, $\pi(\alpha_2)=\beta_{3-i}$. The
remaining case is when $\beta_{3-i}\alpha_2=0$, then $\alpha_2\beta_{3-i}=0$ and we set $\pi(\beta_{3-i})=\alpha$, $\pi(\alpha)=\beta_{3-i}$, $\pi(\beta_i)=\alpha_2$, $\pi(\alpha_2)=\beta_i$.

Case III: 
Let
$\alpha\in Q_1$ be such that $\alpha\beta\neq 0$ for a unique arrow
$\beta$ and $\alpha\beta \in\soc (A)$. Consider $\gamma\beta$ for $\gamma \neq \alpha$, if $\gamma\beta=0,$ we can set $\pi(\alpha)=\beta$. If $\gamma\beta \neq 0$, then there exist a path $p$ and $c \in \kk^*$ such that $p
\gamma\beta-c\alpha\beta=0$, so there is $\beta_2$ such that $(p\gamma-c\alpha)\beta_2\neq 0$. Since $\alpha\beta_2=0$ by assumption $p\gamma\beta_2\neq 0$, so $p$ is a path of length $0$ and we can set $\pi(\alpha)=\beta$, $\pi(\gamma)=\beta_2$.

Now $\pi$ is defined on all $Q_1$ and clearly it is injective
($\pi(x)\neq\pi(y)$ for $x\neq y$ by stably biserial condition if both $x,y$
belong to case I, otherwise $\pi(x)\neq\pi(y)$ by construction).
Then, indeed, $\pi$ is a permutation and it has the following
properties:

  \begin{enumerate}

  \item $\alpha\pi(\alpha)\neq 0$.         \quad\quad\quad (1)

  \item If $\beta\neq \pi(\alpha)$, then $\alpha\beta \in \soc(A)$. \quad\quad\quad (2)

  \end{enumerate}

For any $\alpha\in Q_1$ let $\langle \pi\rangle
\alpha=(\alpha=\alpha_1,\alpha_2,\dots,\alpha_{n_{\alpha}}) $. We define $\alpha_i$ for all natural
$i$ by the condition $\alpha_{i+n_{\alpha}}=\alpha_i$ and find maximal
integer $k_\alpha$ with
$\alpha_1\alpha_2\dots\alpha_{k_\alpha}\neq 0$. Note that $k_\alpha>1$ by
$(1)$, and therefore
 $\alpha_1\alpha_2\dots\alpha_{k_\alpha}\beta=0$ for $\beta\neq\alpha_{k_\alpha+1}$ as well (by $(2)$), i.e. $p_{\alpha}=\alpha_1\alpha_2\dots\alpha_{k_\alpha}\in \scl(A)$.  Actually $p_{\alpha}\in e_{s(\alpha)}Ae_{s(\alpha)}$ by symmetry. Let us define $sc(\alpha)=\alpha_1\dots\alpha_{k_\alpha}$.

\begin{lem}
 \begin{enumerate}
 \item For each $\alpha\in Q_1$ we have $k_{\alpha}=n_{\alpha}m_{\alpha}$ for some integer $m_{\alpha}$.
 \item If $\alpha, \beta\in Q_1$ lie on the common cycle of $\pi$, then $k_{\alpha}=k_{\beta}$ (and $m_{\alpha}=m_{\beta}$).
  \item If $\alpha, \beta\in Q_1$ with $s(\alpha)=s(\beta)$, then $sc(\alpha)=c_{\alpha,\beta}\cdot sc(\beta)$ for some $c_{\alpha,\beta}\in \kk^*$.
  \end{enumerate}
\end{lem}
 We say that $m_{\alpha}$ is the {\it multiplicity of the cycle} $\langle \pi\rangle
\alpha$.

\begin{proof} Put $k=k_{\alpha}$

1. Since  $\alpha_1\alpha_2\alpha_3\dots\alpha_k \in sc(A)$, we have
$\alpha_2\alpha_3\dots\alpha_k\alpha_1\neq 0$. If $k>2$ then
$\alpha_k\alpha_1\notin \soc(A)$. Therefore, by $(2)$,
$\alpha_{k+1}=\alpha_1$ as required. If $k=2$, i.e.
$\alpha_1\alpha_2\in \scl(A)$, then $\alpha$ belongs to Case II or
to Case III and we have $n_{\alpha}=2=k_{\alpha}$.

2. This follows from 1 and from the fact that a socle path cannot be
a subpath of another socle path.

3. It follows from the fact that $\soc (e_{s(\alpha)}A)$ is
one-dimensional.
\end{proof}

Let us call a non-zero path $\beta_1\dots\beta_k$ {\it admissible} if
$\pi(\beta_i)=\beta_{i+1}$ for all $i$. In particular, for any $v\in
Q_0$ we have an admissible path $sc(\alpha)\in\scl(e_vA)$ with
$s(\alpha)=e_v$. So it follows from $(2)$ that any non-zero
non-admissible path is of length $2$ and is equal (in $A$) to an
admissible socle path: $\beta\gamma=k\cdot sc(\alpha)$ for some
$\alpha\in Q_1, k\in \kk^*$. Such an equality we call a {\it socle relation}.
Note that replacing in any socle relation right-hand side by $0$ we
obtain a standard description of SSB-algebra (up to coefficients in
the relations of the form $sc(\alpha)=k\cdot sc(\beta)$, $k\in \kk^*$ but these
coefficients can be eliminated for symmetric algebras).

\begin{lem} In the notations of the previous lemma, we can assume that
$c_{\alpha,\beta}=1$  for all $\alpha, \beta\in Q_1$ with
$s(\alpha)=s(\beta)$ (i.e. $sc(\alpha)=sc(\beta)$).
\end{lem}

\begin{proof} Let $\varphi_{A}(x)=\langle x,1 \rangle$ be induced by the symmetric form $\langle -,- \rangle$ on $A$, put
$c_{\alpha}=\varphi_{A}(sc(\alpha))$. As the form is symmetric,
for $\alpha,\beta$ belonging to the same $\pi$-orbit
$c_{\alpha}=c_{\beta}$, it follows that $c_{\alpha,\beta}=1$ for
such $\alpha,\beta$. Now let $\{\alpha_1,\dots,\alpha_k\}$ be a set
of representatives of $\pi$-orbits. Put
$\alpha_i'=\frac{\alpha_i}{c_i}$, where
$c_i^{m_{\alpha_i}}=c_{\alpha_i}$. Then, replacing 
$\alpha_i$ by  $\alpha_i'$, $1\leq i \leq k$,  for any new socle path $sc(\alpha)'$ we obtain
$\varphi_{A}(sc(\alpha)')=\varphi_{A}(sc(\alpha))/c_i^{m_{\alpha_i}}=1$, where $i$ is defined by $\alpha_i \in \langle \pi \rangle \alpha$.
Therefore, we obtain that if $p_1=kp_2$ for socle paths and $k\neq
0$, then $k=1$ as required. Clearly we have not changed any
relations except for, possibly, changing non-zero coefficients in
socle relations.
\end{proof}

\begin{lem}\label{socrelations} Let $A=\kk Q/I$ be a stably biserial algebra with permutation 
$\pi$, multiplicities $m$ and ideal $I$ generated by the
following relations: 

\begin{enumerate} 
\item $sc(\alpha)-sc(\beta)$ for each $(\alpha,\beta)$ with $s(\alpha)=s(\beta)$. 

\item $sc(\alpha)\alpha$, $\pi^{-1}(\alpha)sc(\alpha)$ for each 
vertex $s(\alpha)$ with one incoming and one outgoing arrow. 
\item $\beta\gamma-l_{\beta,\gamma}sc(\beta)$ for all $\beta\gamma \in Q_1$, $\gamma\neq\pi(\beta)$ ($l_{\beta,\gamma}\in \kk$).

\end{enumerate} 

Consider the ideal $I_1$ obtained from $I$ by replacing generators of the form 
$\beta\gamma-l_{\beta,\gamma}sc(\beta)$ 
by $\beta\gamma$ for $\rm{char}k\neq 2$. If $\rm{char}k= 2$ we make this replacement only in the cases with 
$\beta\neq\gamma$. Then 
$\kk Q/I_1\simeq A$ 

\end{lem} 

\begin{proof} We are going to prove this lemma by induction on the number of 
non-zero $l_{\beta,\gamma}$. Suppose that $l_{\beta_0,\gamma_0}\neq 0$. Put $sc(\beta_0)=\beta_0 p$. Then we 
have $\beta_0(\gamma_0-l_{\beta_{0},\gamma_{0}}p)=0$. Let us 
consider two cases:

1. Suppose that $\beta_0\neq\gamma_0$. Let us show that the substitution $\gamma_0\to\gamma_1$, $\gamma_1=\gamma_0-l_{\beta_{0},\gamma_{0}}p$ 
decreases the number of non-zero $l_{\beta,\gamma}$ (preserving all other relations). 

Looking at the values of
$\varphi_{A}$ we get

$$ 
\varphi_{A}(\gamma_0\beta_0)=\varphi_{A}(\beta_0\gamma_0)=\varphi_{A}(l_{\beta_0,\gamma_0}\beta_0p)=\varphi_{A}(l_{\beta_0,\gamma_0}p\beta_0)\neq 
0. 
$$ 

Let us consider two cases. 

Case I. $\pi(\gamma_0)\neq 
\beta_0$. Then $\gamma_0\beta_0\in sc(A)$, this implies that $\gamma_0\beta_0=l_{\beta_0,\gamma_0}p\beta_0$. 
So in this case we have $\beta_0\gamma_1=0$ and also $\gamma_1\beta_0=0$. 

If $\pi^{-1}(\gamma_0)p=p\pi(\gamma_0)=0$, then 
the substitution $\gamma_0\rightarrow \gamma_1$ clearly does not change any other 
relations and we are done.

If $\pi^{-1}(\gamma_0)p\neq 0$ or $p\pi(\gamma_0)\neq 0$ then $p$ is 
an arrow with $s(p)=s(\gamma_0),e(p)=e(\gamma_0)$ and 
$\pi^{-1}(\gamma_0)$ is an arrow with 
$s(\pi^{-1}(\gamma_0))=s(\beta_0),e(\pi^{-1}(\gamma_0))=e(\beta_0)$ 
(as $\pi^{-1}(\gamma_0)p\in \soc(A)$) and we have $|Q_0|=2$ or 
$|Q_0|=1$. If $|Q_0|=2$, then clearly, $\pi^{-1}(\gamma_0)p\neq 0$ 
implies $p\pi(\gamma_0)\neq 0$ and visa versa. Then the substitution of 
$\gamma_0$ for $\gamma_1$ does not create any new non-zero 
$l_{\beta,\gamma}$. If $|Q_0|=1$ and $Q$ has two loops 
$\alpha,\beta$, with $\pi(\alpha)=\alpha, \pi(\beta)=\beta$, and say 
$\alpha$ plays the role of $\gamma_0$, then 
$\alpha'=\alpha-l_{\alpha,\beta}p$ satisfies the desired relations. 
A coefficient can appear in the relation $sc(\alpha)=c \cdot sc(\beta)$, but 
we can make it equal to 1 as before. Thus, in this case we have changed exactly two relations, obtaining $l_{\beta_0,\gamma_1}=l_{\gamma_1, \beta_0}=0$. 

Case II. $\pi(\gamma_0)=\beta_0$. Then we have $\gamma_0\beta_0\notin sc(A)$ (else we have $\pi(\beta_0)=\gamma_0$ as well). Then $\gamma_1\beta_0=\gamma_0\beta_0-l_{\beta_0,\gamma_0}p\beta_0$, with 
$l_{\beta_0,\gamma_0}p\beta_0\in soc(A)$, and therefore any other path, containing $\gamma_1\beta_0$ is equal to the corresponding path after the substitution $\gamma_1\to\gamma_0$. Also we have $\pi^{-1}(\gamma_0)\gamma_1=\pi^{-1}(\gamma_0)\gamma_0-l_{\beta_0,\gamma_0}\pi^{-1}(\gamma_0)p=\pi^{-1}(\gamma_0)\gamma_0$, as 
$\pi^{-1}(\gamma_0)p$ is of length at least $3$ and $p\neq\gamma_0 p'$ for any path $p'$. By the same reasons $\gamma_1\delta=\gamma_0\delta$ 
where $\delta\neq\beta$, $s(\delta)=s(\beta)$. Thus, in this case we have changed exactly one relation, obtaining $l_{\beta_0,\gamma_1}=0$.

2. Suppose $\rm{char}\kk\neq 
2$ and $\beta_0=\gamma_0$, $|Q_0|\neq 1$. In this case 
$s(\beta_0)=e(\beta_0)$, $p$ is a path of length more than $1$ (else 
we have two loops at one vertex), $\beta_0p=p\beta_0\in sc(A)$. Put 
$\beta_0'=\beta_0-l_{\beta_0,\gamma_0}p/2$. Then 
$(\beta_0')^2=(\beta_0-p/2)^2=\beta_0^2-l_{\beta_0,\gamma_0}\beta_0p-l_{\beta_0,\gamma_0}p\beta_0+0=0$. 
As $\alpha p=p \alpha=0$ for all arrows $\alpha\neq\beta_0$ ($p$ 
is not an arrow), all other relations are preserved. 

If $|Q_0|=1$ and $p$ is a path of length more than $1$, the proof 
goes similar. If $p$ is a path of length $1$, by construction of $\pi$ we have
$p^2=0$ and lemma also holds. 
\end{proof}

By Lemma \ref{socrelations} and induction on the number of
non-zero $l_{\beta,\gamma}$ we get the following theorem:

\begin{thm}
1. Any symmetric stably biserial algebra over an algebraically
closed field $\kk$ with $\rm{char} \kk\neq 2$ is isomorphic to a special
biserial algebra.

2. Consider a standard description of a symmetric special biserial
algebra $A=\kk Q/I$ and any set of loops $\{\alpha_1,\dots,\alpha_k\} $
in $Q_1$, where $\pi(\alpha_i)\neq\alpha_i$ for all $i$ (so that
$\alpha_i^2=0$ in $A$), consider a set $\{c_{\alpha_1},\dots,c_{\alpha_k}\} $, $c_{\alpha_i}\in \kk^*$. Replacing in the standard set of relations
$\alpha_i^2$ by $\alpha_i^2-c_{\alpha_i}sc(\alpha_i)$ we obtain a new algebra
$A'$ and all stably biserial algebras can be obtained in this way.

\end{thm}

\end{document}